\documentclass[draft,11pt, a4paper, sort&compress]{elsarticle}

\usepackage{amsmath}
\usepackage{amssymb}
\usepackage{amsthm}
\usepackage{mathrsfs}
\usepackage{amsfonts}
\usepackage{enumerate}
 \usepackage[protrusion=true,expansion=true]{microtype}
 \usepackage[numbers]{natbib}

   \usepackage[varg]{txfonts}
 
\usepackage[hmarginratio=1:1,top=32mm,columnsep=20pt]{geometry}

\newtheorem{theorem}{Theorem}[section]
\newtheorem{lemma}[theorem]{Lemma} 
\newtheorem{corollary}[theorem]{Corollary}
\newtheorem{proposition}[theorem]{Proposition}
\newtheorem{remark}[theorem]{Remark}
 
\numberwithin{equation}{section}

 \newcommand{\eps}{\varepsilon}
\newcommand{\norm}[1]{ \Vert#1 \Vert}
\newcommand{\abs}[1]{\left\vert#1\right\vert}
\newcommand{\set}[1]{\left\{\,#1\,\right\}}
\newcommand{\inner}[1]{\left(#1\right)}
\newcommand{\comi}[1]{\left<#1\right>}

\makeatletter
\def\ps@pprintTitle{%
     \let\@oddhead\@empty
     \let\@evenhead\@empty
     \def\@oddfoot{\reset@font\hfil\thepage\hfil}
     \let\@evenfoot\@oddfoot}
 \makeatother

\begin{document}

 \begin{frontmatter}


\title{Analytic smoothing effect
of the spatially inhomogeneous 
Landau equations
for hard potentials}

\author[ad0]{Hongmei Cao}
 \ead{hmcao\_91@nuaa.edu.cn}

 \author[ad1,ad2]{Wei-Xi Li}
 \ead{wei-xi.li@whu.edu.cn}

 \author[ad0]{Chao-Jiang Xu}
 \ead{xuchaojiang@nuaa.edu.cn}

 \address[ad0]{College of Mathematics and Key Laboratory of Mathematical MIIT,\\  
 Nanjing University of Aeronautics and Astronautics, Nanjing 210016, China }

 \address[ad1]{School of Mathematics and Statistics, Wuhan University, Wuhan
 430072, China}
 \address[ad2]{Hubei Key Laboratory of Computational Science, Wuhan University, 430072 Wuhan, China}

 \begin{abstract}
  We study the spatially inhomogeneous Landau equations with hard potential in the perturbation setting, and establish the analytic smoothing effect in both spatial and velocity variables  for a  class of  low-regularity weak solutions.  This  shows  the Landau equations behave essentially  as  the hypoelliptic Fokker-Planck operators.  The spatial analyticity relies on a new time-average operator, and  the proof is based on a straightforward energy estimate with a careful estimate on the derivatives with respect to the new time-average operator.
 \end{abstract}


\end{frontmatter}

 \tableofcontents

\section{Introduction and main result}

As  specific degenerate elliptic  operators, we may expect not only  the usual $C^\infty$ smoothness but also the Gevrey regularity  for subelliptic operators (cf. Derridj-Zuily \cite{MR512976}), recalling Gevrey class is an intermediate space between $C^{\infty}$ and analytic spaces. However  it is highly non-trivial to improve the  Gevrey regularity to analyticity for degenerate elliptic operators. In 1972, Baouendi-Goulaouic \cite{MR296507}  constructed a  counterexample that shows analytic-hypoellipticity may fail for some  degenerate elliptic
              operators. Since then  it is  a long-standing problem to explore the sufficient and necessary conditions for the analytic regularity of degenerate operators, and so far there have been extensive related works  with  various applications in PDEs and complex analysis; interested readers may refer to D.Tartakoff's monograph \cite{MR2809628} and the references therein, for the comprehensive presentation on the Treves's conjecture and the analyticity of $\bar\partial$-Neumann problem.
      As a positive example  and  a classical subelliptic operator,   the Kolmogorov operator indeed enjoy the analyticity regularity in the degenerate direction. Thus it is natural to ask the similar properties for the spatial inhomogeneous kinetic equations since these equations may be regarded as  non-local and non-linear models of Kolmogorov operators.  However up to now the analytic  regularization is far from well-understood  for the spatially inhomogeneous kinetic equations  and it is only verified  recently  by Morimoto-Xu \cite{MR4147427} for the Landau equations in the Maxwellian molecules case. Note the proof in \cite{MR4147427} relies on the specific structure of  collision operators in the Maxwellian molecules case, where  the Fourier analysis techniques  work  well but can not apply to the cases with hard or soft potentials.  In this text we aim to extend  the analyticity properties established by Morimoto-Xu \cite{MR4147427} to the case of hard potential potentials, and   instead of  the  estimates on the usual derivatives,   our proof relies crucially on a careful treatment on the derivatives with respect to a new time-average operator.

Denoting  by $F=F(t,x,v)\geq 0$   the density of particles with
velocity $v\in\mathbb R^3$ at time $t\geq0$ and position $x\in\mathbb T^3$, the spatially  inhomogeneous Landau equation reads
\begin{equation}\label{Landau}
\begin{cases}
\partial_t F+v\cdot\partial_x F=Q_L(F,\ F),\\
F|_{t=0}=F_0,
\end{cases}
\end{equation}
where
\begin{equation}\label{colli}
Q_L(G,H)=\sum_{1\leq i, j\leq 3}\partial_{v_i} \set{\int_{\mathbb R^3}a_{i,j}(v-v_*)\big[G(v_*)\partial_{v_j}
H(v)-H(v)(\partial_{v_j}G)(v_*)\big]dv_*},	
\end{equation}
with $(a_{i,j})_{1\leq i,j\leq 3}$   a nonnegative definite   matrix given by \begin{equation}\label{coe}
a_{i,j}(v)=\inner{\delta_{ij}|v|^2-{v_iv_j}}\abs v^{\gamma},\quad \gamma \in ]-3,1].
\end{equation}
Here and below $\delta_{ij}$ is the Kronecker delta function.
 By   spatially  inhomogeneous   it means the unknown density $F$ depends on the spatial  $x$ variable. Meanwhile if $F$ is independent of $x$ it is then called spatially homogeneous. In this paper we are concerned with so-called hard potentials  that  means  $0<\gamma\leq 1 $ in  \eqref{coe}.    The Maxwellian molecules case of $\gamma=0$ was investigated recently  by Morimoto-Xu \cite{MR4147427}.  We remark that our argument may cover the specific case of Maxwellian molecules, but can not apply to the case of soft potentials which means $-3<\gamma<0$ in \eqref{coe}.

The collisional operator $Q_L$ only acts velocity variable $v$ and  roughly speaking it behaves as the Laplacian $\Delta_v$ and thus admits a similar  diffusion properties as heat equations (cf.Desvillettes-Villani  \cite{MR1737547}). The similar properties for Boltzmann equations was established by
 Alexandre-Desvillettes-Villani-Wennberg \cite{MR1765272} in the form of entropy dissipation estimates.   So it is a natural conjecture that the spatially homogeneous Landau equations should enjoy a similar smoothing effect as heat equations, and this has been well confirmed in various settings; we refer to \cite{MR2557895, MR2523694, MR3177625} for the sharp smoothing effect  in analytic or ultra-analytic setting, after the earlier works of \cite{MR2514370,MR2425602, MR1737547}.  However there are only few works on the sharp regularity  of spatially Boltzmann, and here we only mention the work  \cite{MR3665667} of Barbaroux-Hundertmark-Ried-Vugalter where they established the sharp Gevrey class regularization for the specific Maxwellian molecules case  with the counterpart  properties for the other soft or hard  potentials remaining unsolved. We refer the interested readers to
   \cite{MR3572500, MR3325244, MR3310275, MR2465814, MR2959943, MR2038147, MR3485915} for the relevant works on the regularity issue for spatially homogeneous Boltzmann equations.

   Compared with the  spatially  homogeneous kinetic equations,  much less is known for the higher order regularity of weak solutions in  the spatially inhomogeneous case. The main difficulty  arises from the spatial degeneracy coupled with the nonlocal collision term, and it may ask  more subtle analysis to treat the nonlinear  collision operators involving rough coefficients.
For general initial data with finite mass, energy and entropy,  the global existence of   renormalized weak solutions in $L^1$ was established first by DiPerna-Lions \cite{MR1014927} for the Boltzmann equations under the Grad's angular cutoff assumption,  and  later by Alexandre-Villani \cite{MR1857879}   for the non-cutoff   case,  while both uniqueness and regularity of such general global solutions are still a challenging open problem.  Here we mention the recent progress  \cite{MR4033752,2020arXiv200502997I,2019arXiv190912729I,MR4229202,MR4049224,MR3551261, MR4072211, MR3778645} toward the conditional regularity for  Boltzmann and Landau equations with general initial data.  It would be interesting to develop a self-contained theory of both existence and regularity without any extra condition on solutions.

 In the perturbation setting,  the existence and unique theory of mild solutions is well-explored for the Boltzmann equations.  For exponential perturbations   near Maxwellians,  when the initial data belong to some kind of  regular Sobolev spaces,  the well-posednss  was established independently  by two groups AMUXY \cite{MR2679369,MR2793203,MR2795331,MR2863853} and Gressman-Strain \cite{MR2784329}.    Furthermore,   the $C^\infty$-smoothing effect was proven by   \cite{MR2795331, MR2679369},  and   a higher order  Gevrey regularity,  inspired by the behaviors of Kolmogorov operators, was proven by Lerner-Morimoto-Pravda-Starov-Xu \cite{MR3348825} and \cite{MR4375857} for the1D and 3D Boltzmann equations, respectively.   Recently the existence and uniqueness of some mild weak solutions was established by Duan-Liu-Sakamoto-Strain \cite{MR4230064} and  the Gevrey regularity were proven by \cite{MR4356815}.  For the perturbation setting with polynomial decay,  the classical solutions  were constructed by Alonso-Morimoto-Sun-Yang \cite{MR4201411}; see also the independent works of He-Jiang \cite{2017arXiv171000315H} and H\'erau-Tonon-Tristani \cite{MR4107942}.   The unique existence of   weak solutions in the $L^2\cap L^\infty$ setting was initiated by  Alonso-Morimoto-Sun-Yang  \cite{2020arXiv201010065A}, and we mention the recent Silvestre-Snelson's work  \cite{2021arXiv210603909S}.  The well-posedness  for  Landau equations can be found in \cite{MR1946444,MR4230064, MR3625186, MR3483898,MR3984752} and references therein

 Finally we mention two  techniques  used frequently when  investigate the  regularity property  of  kinetic equations,    one referring to De Giorgi-Nash-Moser theory  with the help of the averaging lemma and another to H\"ormander's hypoelliptic theory.  Interested readers may refer to \cite{MR3923847,2020arXiv201010065A,MR4033752,2020arXiv200502997I,2019arXiv190912729I,MR4229202,MR4049224,MR3551261, MR4072211, MR3778645} for  the De Giorgi type argument,   and  to \cite{MR4375857,MR4356815,MR3950012,MR2763329,MR2467026,MR3102561,MR3193940,MR2786222,MR2034753,MR3348825,MR1949176} for the application of hypoelliptic techniques to kinetic equations. More details on the two techniques can be found in
    the survey paper of C. Mouhot \cite{MR3966858}.

In this work we are concerned with the sharp regularity of weak solutions to Landau equations, and this is  inspired by  a natural conjecture which expects  kinetic equations admit a similar smoothing effect as that for heat or Kolmogorov operators. Compared with the spatially homogeneous case,  it is far from well-understood  for the inhomogeneous case since the degeneracy occurs in the spatial  variable   and thus it is  a hypoelliptic problem due to the  non-trivial interaction between the transport operator and the collision operator. In general  we can expect only the  Gevery regularization  effect  for hypoelliptic operators. On the other hand, a straightforward Fourier analysis shows the hypoelliptic  Fokker-Planck operator,   a local model of spatially inhomogeneous Landau equations, enjoys the analyticity regularity for both  spatial and velocity variables.    However the analytic regularization for the spatially inhomogeneous kinetic equations is still unclear,  except some specific models; here we mention the recent positive results  in Morimoto-Xu \cite{MR4147427} on Landau equations with Maxwellian molecules.  The argument in \cite{MR4147427}  depends    on the Fourier analysis  that does not apply to the cases of hard  and soft potentials.    In this work we use a more flexible $L^2$ energy method to analyze the sharp smoothing effect for the spatially inhomogeneous Landau equations with hard potentials,  and we hope the argument presented here may help give insight on the optimal regularity of  Boltzmann equations.  Recall for Boltzmann equations we only have Gevrey  regularity of    index $(1+2s)/2s$  in both spatial and velocity variables (cf. \cite{MR4375857,MR3348825,MR4356815}),  which seems not optimal in view of the counterpart for Kolmogorov type operators with fractional  diffusion in velocity.

We will restrict our attention to the fluctuation around the Maxwellian distribution $\mu$ with
\begin{equation*}
\mu(v)=(2\pi)^{-3/2}e^{-|v|^{2}/2}.
\end{equation*}
Write solutions $F$ to \eqref{Landau} as $F=\mu+\sqrt{\mu}f$ and accordingly $F_0=\mu+\sqrt{\mu}f_0$ for  initial data. Then the fluctuation $f$ satisfies the Cauchy problem
\begin{eqnarray}\label{cau}
\begin{cases}
\partial_{t}f+v\cdot\partial_{x}f+\mathcal{L} f
      =\Gamma(f, f),    \\
f|_{t=0}=f_0,
\end{cases}
\end{eqnarray}
where here and below we use the notations
\begin{equation}\label{gm}
\mathcal{L} f=-\Gamma(\sqrt\mu, f)-\Gamma(f,\sqrt\mu),\quad
\Gamma (g,h)=\mu^{-{1\over 2}}Q_L(\sqrt{\mu}g,\sqrt{\mu}h).
\end{equation}
We recall the space $L^1_m  L^2_v$  introduced in \cite{MR4230064},  that  consists of all functions $u$ such that $\norm{u}_{L^1_m   L^2_v}<+\infty$   with
\begin{equation}\label{l1norm}
\norm{u}_{L^1_m   L^2_v}:=\sum_{m\in\mathbb Z^3}
  \norm{\mathcal F_x u (m, \cdot)}_{L^2_v},
\end{equation}
where  $\mathcal F_x  u (m,v)$ stands for the partial Fourier transform of $u(x,v)$ with respect to $x\in\mathbb T^3$.  Here and throughout the paper $m\in\mathbb Z^3$ stands for the Fourier dual variable of $x\in\mathbb T^3$.

With the above notation the  main result can be stated as follows.

\begin{theorem}\label{thm:main}
Let $\gamma \ge 0$ in \eqref{coe} and let $L_m^1L_v^2$ be defined by \eqref{l1norm}.  Suppose  the initial datum  $ f_0 $ of \eqref{cau} belongs to $ L^1_mL^2_v$ such that \begin{equation}\label{123}
 \norm{f_0}_{  L^1_mL^2_v} \leq
 \varepsilon_0,
\end{equation}
for some constant $\varepsilon_0>0$.  If $\eps_0$ is small sufficiently, then the Cauchy problem \eqref{cau}  admits an  unique global-in-time solution satisfying that a constant $C$ exists  such that  the following estimates hold true:
\begin{equation}\label{decayana}
\forall\ \alpha,\beta\in\mathbb Z_+^3 \textrm{  with  }  \abs\alpha\geq1,\quad 		\sup_{t>0} \tilde t^{\frac{3}{2}\abs\alpha+ \frac{1}{2}\abs\beta}\norm{\partial_x^\alpha \partial_{v}^\beta f(t)}_{L^2} \leq  \eps_0   C^{ \abs\alpha+ \abs\beta+1}  (\abs\alpha+\abs\beta)!,
\end{equation}
and
\begin{equation}
\label{decayrate}
\forall\  \beta\in\mathbb Z_+^3,\quad 		\sup_{t>0} \tilde t^{  \frac{\abs\beta}{2} +\frac{3}{2}}\norm{\partial_{v}^\beta f(t)}_{L^2} \leq  \eps_0   C^{   \abs\beta+1}  \abs\beta !,
\end{equation}
 where  $\tilde t=\min\{1, t\}$. In particular,  the function $(x, v)\mapsto f(t,x,v)$ is real analytic in $\mathbb T^3\times\mathbb R^3$ for all positive time $t>0$.
\end{theorem}

\begin{remark}
The global-in-time existence and unique of  solutions  in regular Sobolev space was proven by  Y.Guo \cite{MR1946444}. Meanwhile the  unique global solution in   $L_m^1L_v^2$ was established recently by Duan-Liu-Sakamoto-Strain \cite{MR4230064}. This text aims to improve the weak $L_m^1L_v^2$ regularity  to analyticity at positive time.
 \end{remark}

 \begin{remark}
 	The short time decay  in \eqref{decayrate}  for velocity estimate is not sharp, which is caused by the norm $\norm{\cdot}_{H_x^2L_v^2}$  used to deal with  trilinear terms.  If   using the norm $\norm{\cdot}_{L^1_m   L^2_v}$ in \eqref{l1norm} instead, we may expect the sharp short time decay for velocity estimates, that is, the estimate \eqref{decayana} will hold true for all $\alpha\in\mathbb Z_+^3$ without the restriction that $\abs\alpha\geq 1.$
 	   \end{remark}

   \noindent {\bf Notations.} For simplicity of notations, we will use  $\norm{\cdot}_{L^2}$ and $\inner{\cdot, \cdot}_{L^2}$ to denote the norm and inner product of  $L^2=L^2(\mathbb T_x^3\times\mathbb R_v^3)$   and use the notation   $\norm{\cdot}_{L_v^2}$ and $\inner{\cdot, \cdot}_{L_v^2}$  when the variable $v$ is specified. Similar notation  will be used for $H^{+\infty}=H^{+\infty}(\mathbb T_x^3\times\mathbb R_v^3)$. In addition,  We denote by $H^{(2,0)}=H^{2}_x(L^2_v)$   the  classical Sobolev space, that is
\begin{equation*}
H^{(2,0)}=\big\{u \in L^2(\mathbb T_x^3\times \mathbb R_{v}^{3}) ; \,\,\, \partial^\alpha_x u \in L^2(\mathbb T_x^3\times \mathbb R_{v}^{3}), \, |\alpha|\leq 2\big \},
\end{equation*}
which is complemented with the norm $\norm{\cdot}_{(2,0)}$ and   product $\inner{\cdot,\ \cdot}_{(2,0)}$, that is,
\begin{equation}\label{norin}
\|u\|_{(2,0)}=\Big(\sum_{\abs\alpha\leq 2}\|\partial_x^\alpha u\|_{L^2}^2\Big)^{1/2},\quad (u,w)_{(2,0)}=\sum_{\abs \alpha\leq2} \inner{\partial_x^\alpha u,\ \partial_x^\alpha w}_{L^2}.
\end{equation}
Throughout the paper $\hat u (m,\eta)$ stands for the Fourier transform of $u$ with respect to $(x,v)$, with $(m,\eta)\in \mathbb Z^3\times\mathbb R^3 $ the Fourier dual variables of $(x,v)\in \mathbb T^3\times\mathbb R^3$.

For a vector-valued function $A=(A_1,A_2, \ldots, A_n)$, we used the convention that $\norm{A}^2=\sum_{1\leq j\leq n}\norm{A_j}^2$ for a generic norm $\norm{\cdot}.$  Moreover $\comi v=\inner{1+v^2}^{1/2}$.

 Denote by $[T_1, T_2]$ the commutator between two operators $T_1$ and $T_2$, that is,
\begin{equation}\label{defcom}
	[T_1, T_2]=T_1T_2-T_2T_1.
\end{equation}
We denote by $v\wedge\eta$ the cross product  of two vectors $v=(v_1,v_2,v_3)$ and  $\eta=(\eta_1,\eta_2,\eta_3)$ which is defined by
\begin{equation}\label{crosec}
	v\wedge \eta=(v_2\eta_3-v_3\eta_2, v_3\eta_1-v_1\eta_3,  v_1\eta_2-v_2\eta_1).
\end{equation}
Finally, in the following discussion   $g*h$ stands for the convolution only with respect to the velocity variable  $v$
for two functions $g$ and $h$, that is,
\begin{eqnarray*}
	g*h(x, v)=\int_{\mathbb R^3} g(x,v-v_*) h(x, v_*) dv_*.
\end{eqnarray*}

\section{Analytic regularity of  smooth solutions}

In this section we improve the  $H^\infty$-smoothness  of solutions  to analyticity, and the analytic regularization effect of  $L_m^1L_v^2$ weak solutions  is postponed the the last section.

 In the following discussion we let $t_0 \in ]0, 1/2] $ be an arbitrarily  fixed time, and introduce  a  time-average differential operator $M$ by setting
 \begin{equation}\label{timave}
M
= -\int^t_{t_0}|\partial_{v_1}+(r-t_0) \partial_{x_1}|^2 dr
=-  (t-t_0)\partial_{v_1}^2-(t-t_0)^2 \partial_{x_1}  \partial_{v_1}-\frac{(t-t_0)^3}{3}\partial_{x_1}^2,
\end{equation}
which plays  a crucial role when investigating  the analytic regularity.
Note $M$ is a Fourier multiplier with symbol
\begin{eqnarray*}
	 (t-t_0)\eta_1^2+ (t-t_0)^2m_1\eta_1+\frac{ (t-t_0)^3}{3}m_1^2,
\end{eqnarray*}
that is,
\begin{eqnarray*}
		\widehat {M  g}(m,\eta)=\Big( (t-t_0)\eta_1^2+ (t-t_0)^2m_1\eta_1+\frac{ (t-t_0)^3}{3}m_1^2\Big)\hat g(m,\eta),
\end{eqnarray*}
recalling  $(m,\eta)\in \mathbb Z^3\times\mathbb R^3 $ are the Fourier dual variables of $(x,v)\in \mathbb T^3\times\mathbb R^3$.  The key observation  is that the spatial derivatives are not involved in  the commutator between   $M$  and the transport operator. Precisely, direct verification shows
  \begin{equation}\label{keyob}
  	[M, \,\,  \partial_t+v\,\cdot\,\partial_x ]=\partial_{v_1}^2,
  \end{equation}
  recalling $[\cdot,\, \cdot]$ stands for the commutator between two operators defined in \eqref{defcom}.
  So that we may make use of the diffusion property in velocity direction, to obtain the spatial  analyticity.

   \begin{theorem}\label{thm:key} Suppose    $0<t_0\leq 1/2 $ is an arbitrarily  given time.   Let $\gamma\geq 0$ in  \eqref{coe} and  let $f\in L^\infty\big([t_0,1];\ H^{+\infty}\big)$ be any solution to the   Landau equation \eqref{cau}. With the notations given at the end of the previous section,  we
      suppose that
       \begin{equation}\label{regass}
  		\sup_{t_0\leq t\leq 1} \norm{  f(t)}_{(2,0)}+\bigg(\int_{t_0}^1 \norm{\psi(v,D_v)   f(t)}_{(2,0)}^2dt\bigg)^{1\over2} \leq    \epsilon
  \end{equation}
  for some constant $\epsilon>0$ and that
   \begin{equation}\label{hinfty}
   \forall\ \alpha,\beta\in\mathbb Z_+^3, \quad   	\sup_{t_0\leq t\leq 1} \norm{\partial_x^\alpha\partial_v^\beta f(t)}_{(2,0)}+\bigg(\int_{t_0}^1 \norm{\psi(v,D_v) \partial_x^\alpha\partial_v^\beta f(t)}_{(2,0)}^2dt\bigg)^{1\over2} <+\infty,
   \end{equation}
   where and below
   \begin{equation}\label{ma+0}
   	\norm{\psi(v,D_v)g}_{(2,0)}^2:=\norm{\comi v^{\gamma\over 2}  \partial_v g}_{(2,0)}^2+\norm{\comi v^{\gamma\over 2} (v\wedge\partial_v)g}_{(2,0)}^2+\norm{\comi v^{1+{\gamma\over 2}}  g}_{(2,0)}^2
   \end{equation}
   with  $v\wedge\partial_v$ and $\norm{\cdot}_{(2,0)}$ defined  by \eqref{crosec} and  \eqref{norin} respectively, and $\comi v=(1+v^2)^{1/2}$.
 If
    $\epsilon$ is small sufficiently, then there exists a constant $C_*\geq 1$ independent of $\epsilon$, such that \begin{equation}\label{keyestimate}
\forall\ k\in\mathbb{Z}_+,\quad\sup_{t_0\leq t\leq 1} \norm{   M^{k}f(t)}_{(2,0)}+\Big(\int_{t_0}^1 \norm{\psi(v,D_v)    M^{k} f(t)}_{(2, 0)}^2dt\Big)^{1\over2} \leq  \frac{\epsilon}{(2k+1)^3}  C_*^{2k}  (2k)!,
\end{equation}
where $M$ is defined by \eqref{timave}.
 Moreover, the above estimate \eqref{keyestimate} still holds true if we
   replace $   M$  by
\begin{eqnarray*}
	-  (t-t_0)\partial_{v_i}^2-(t-t_0)^2 \partial_{x_i}  \partial_{v_i}-\frac{(t-t_0)^3}{3}\partial_{x_i}^2,\quad i=2 \textrm { or } 3.
\end{eqnarray*}
\end{theorem}

To prove Theorem \ref{thm:key}, we need the following    two propositions.

 \begin{proposition}[Trilinear and coercivity estimates]\label{prop: tricoe} Let $\Gamma(g,h)$ and $\mathcal L$ be defined in \eqref{gm}.
 There exists a constant $C_1$ such that for any  $g,h,\omega\in H_v^{+\infty}$ with $\psi(v,D_v)h,  \psi(v,D_v)\omega\in L_v^2$, we have the trilinear estimate
 	\begin{equation}\label{trili}
 	\big|\big(\Gamma(g,h), \ \omega \big)_{L_v^2} \big|:=	\Big|\int_{\mathbb R^3}\Gamma(g, h)\,  \omega dv\Big|  \leq C_1\norm{g}_{L_v^2}\norm{\psi(v,D_v)h}_{L_v^2}\norm{\psi(v,D_v)\omega}_{L_v^2},
 	\end{equation}
 	and the coercivity estimate
 	\begin{equation}\label{coer}
 		\norm{\psi(v,D_v)h}_{L_v^2}^2\leq C_1\inner{ \mathcal Lh, \, h }_{L_v^2} +C_1 \norm{h}_{L_v^2}^2.
 	\end{equation}
 \end{proposition}

 \begin{proposition}
	[Commutator estimate]\label{prop:com}Let $k\geq 1$ be a given integer  and  $f\in L^\infty([t_0,1]; H^{+\infty})$ be any solution to \eqref{cau} satisfying that a constant $C_*\geq 1$ exists such that
  \begin{equation}\label{ass:ind}
\forall \ j\leq k-1,\quad \sup_{t_0\leq t\leq 1} \norm{ M^j f(t)}_{(2,0)}+\Big(\int_{t_0}^1 \norm{\psi(v,D_v) M^j f(t)}_{(2,0)}^2dt\Big)^{1\over2} \leq \frac{\epsilon}{(2j+1)^3}C_*^{2j}  (2j)!.
\end{equation}
If $C_*$ is large enough,
then there exists a contant $C_2>0$ independent of $\epsilon$ and $C_*$ above,  such that for any $\delta>0$,
\begin{eqnarray*}
\begin{aligned}
	&\int_{t_0}^1\big|\left(M^k\Gamma(f, f)-\Gamma(f, M^kf), M^k f\right)_{(2,0)}\big|dt+\int_{t_0}^1\big|\left([M^k,\,\mathcal{L}]\, f,\, M^k f\right)_{(2,0)}\big| dt \\
&	\leq \inner{ \delta +\epsilon C_2}\bigg[ \Big(\sup_{t_0\leq t\leq 1} \norm{M^{k} f }_{(2,0)}\Big)^2+  \int_{t_0}^1\norm{ \psi(v,D_v)  M^{k}   f}_{(2,0)}^2dt   \bigg]
	  +  C_\delta \bigg[ \Big(\epsilon+C_*^{-1}\Big) \frac{\epsilon}{(2k+1)^3}  C_*^{ 2k} (2k)!  \bigg]^2,
	  \end{aligned}
\end{eqnarray*}
where and below $C_\delta$ stands for   generic constants depending on $\delta.$ Recall $[\cdot, \cdot]$ stands for the commutator   defined by \eqref{defcom}.
\end{proposition}

\begin{remark}
	 In the above proposition, by saying  $C_*$ is large enough we mean the condition  on $C_*$ required  in Lemma \ref{lem:tec0} is fulfilled.
\end{remark}

\begin{remark}\label{rem: well-def}
	As to be seen in Proposition \ref{proprep},  if $f\in L^\infty\big([t_0,1];  H ^{+\infty}\big)$ such that $\psi(v,D_v) f\in L^2\big([t_0, 1]; H^{+\infty}\big)$,  then  for any $k\in\mathbb Z_+,$
	\begin{eqnarray*}
	\comi{v}^{-(1+\frac{\gamma}{2})} M^k\Gamma(f, f), \, \comi{v}^{-(1+\frac{\gamma}{2})}  \Gamma(f, M^kf),\, \comi{v}^{ 1+\frac{\gamma}{2} }M^kf\in L^2\big([t_0,1];  H^{+\infty}\big).
	\end{eqnarray*}
	Thus
	\begin{eqnarray*}
	\begin{aligned}
		\int_{t_0}^1 \left(M^k\Gamma(f, f),  \, M^k f\right)_{(2,0)} dt:&=	 \sum_{\abs\alpha\leq2} \int_{t_0}^1 \int_{\mathbb T_x^3\times \mathbb R_v^3} \Big[ \partial_x^\alpha M^k\Gamma(f, f) \Big] \    \overline{ \partial_x^\alpha M^k f }\ dxdv dt\\
		&=	 \sum_{\abs\alpha\leq2}\int_{t_0}^1 \int_{\mathbb T_x^3\times \mathbb R_v^3} \Big[ 	\comi{v}^{-(1+\frac{\gamma}{2})}  M^k\Gamma(f, f) \Big]\  \overline{	\comi{v}^{ 1+\frac{\gamma}{2}}   \partial_x^\alpha M^k f }\ dx dv dt
		\end{aligned}
	\end{eqnarray*}
	is well-defined, and  so are  the other  trilinear or quadratic terms  in Propositions \ref{prop:com} and \ref{prop: tricoe}.
\end{remark}

The proofs  of Propositions \ref{prop: tricoe} and \ref{prop:com} are quite lengthy and we postpone them to the next two sections.   By virtue of the two propositions above we are enable  to complete the proof of Theorem \ref{thm:key}.

  \begin{proof}[Proof of Theorem \ref{thm:key}]
	We use induction on $k$ to derive the estimate in Theorem \ref{thm:key}.  The validity of \eqref{keyestimate} for $k=0$ is obvious in view of \eqref{regass}.  Now for given $k\geq 1,$ suppose
 \begin{equation}\label{+ass:ind}
\forall \ j\leq k-1,\quad \sup_{t_0\leq t\leq 1} \norm{ M^j f(t)}_{(2,0)}+\Big(\int_{t_0}^1 \norm{\psi(v,D_v) M^j f(t)}_{(2,0)}^2dt\Big)^{1\over2} \leq \frac{\epsilon}{(2j+1)^3}C_*^{2j}  (2j)!
\end{equation}
for some constant $C_*\geq 1$ to be determined later.
We will show the validity of \eqref{+ass:ind} for $j=k$.
	
Applying $M^k$ to \eqref{cau} yields
	\begin{equation*}
	  \partial_t M^kf+v\cdot\partial_x  M^kf +M^k \mathcal L f=	M^k\Gamma(f, f)-[M^k,\, \partial_t+v\,\cdot\,\partial_x ]\, f=	M^k\Gamma(f, f)-k\partial_{v_1}^2M^{k-1} f,
	\end{equation*}
 the last equality using \eqref{keyob}, that is,
 \begin{equation}\label{mkequ}
	  \partial_t M^kf+v\cdot\partial_x  M^kf + \mathcal LM^k f =	\Gamma(f,\ M^kf)+M^k\Gamma(f, f)-\Gamma(f,\ M^kf)- \big[M^k,\ \mathcal L\big] f -k\partial_{v_1}^2M^{k-1} f.
	\end{equation}
   In view of \eqref{hinfty},  we can verify directly that
  \begin{eqnarray*}
 	v\cdot \partial_x M^kf,\, \partial_{v_1}^2M^{k-1} f\in L^2\big([t_0,1]; H^{(2, 0)}\big) \ \textrm{ and }\ \lim_{t\rightarrow t_0}\norm{M^k f(t)}_{(2,0)}\leq C \lim_{t\rightarrow t_0}   (t-t_0)  \norm{ f }_{L^\infty\inner{[t_0,1]; H^{2k}}}  =0 .
 \end{eqnarray*}
This, with  \eqref{mkequ} and Remark \ref{rem: well-def}, implies
\begin{equation}\label{commut-1}
\begin{aligned}
&\frac{1}{2} \|M^k f(t)\|^2_{(2,0)}+\int_0^t  \inner{\mathcal L M^{k} f,\, M^k f}_{(2,0)} ds \\
&\leq  k\int_{t_0}^1\|\partial_vM^{k-1} f\|_{(2,0)} \|\partial_vM^{k} f\|_{(2,0)}  dt+\int_{t_0}^1 \big|\left(\Gamma(f, M^k f), M^k f\right)_{(2,0)}\big|dt
\\&\quad
+\int_{t_0}^1 \big|\left(M^k\Gamma(f, f)-\Gamma(f, M^kf), M^k f\right)_{(2,0)}\big|dt+\int_{t_0}^1\big|\left([M^k,\,\mathcal{L}]\, f,\, M^k f\right)_{(2,0)}\big|dt.
\end{aligned}
\end{equation}	
In view of the assumption \eqref{hinfty},  we use the coercivity  \eqref{coer} to conclude
\begin{align*}
 \int_0^t  \norm{\psi(v,D_v)M^k f}_{(2,0)}^2 dt   \leq C_1 \int_0^t  \inner{\mathcal L M^{k} f,\, M^k f}_{(2,0)} dt	+ C_1\int_0^t \norm{ M^k f}_{(2,0)}^2 dt.
\end{align*}
As for the terms on the right hand side of \eqref{commut-1},
it follows from the trilinear estimate \eqref{trili} and the inductive assumption  \eqref{+ass:ind} that
\begin{multline*}
	\int_{t_0}^1 \big|\left(\Gamma(f, M^k f), M^k f\right)_{(2,0)}\big|dt\\ \leq \big(\sup_{t_0\leq t\leq 1}\norm{f(t)}_{(2,0)} \big) \int_{t_0}^1 \norm{ \psi(v,D_v)M^k f}_{(2,0)}^2  dt
	\leq  \epsilon\int_{t_0}^1 \norm{ \psi(v,D_v)M^k f}_{(2,0)}^2  dt.
\end{multline*}
Moreover, using again  the inductive assumption \eqref{+ass:ind} yields, for any $\delta>0,$
\begin{align*}
	&k\int_{t_0}^1\|\partial_vM^{k-1} f\|_{(2,0)} \|\partial_vM^{k} f\|_{(2,0)}  dt \\
	&\leq \delta  \int_{t_0}^1 \|\psi(v,D_v)M^{k} f\|_{(2,0)}^2  dt+\frac{ k^2}{\delta}\int_{t_0}^1\|\psi(v,D_v)M^{k-1} f\|_{(2,0)}   ^2 dt\\
	&\leq \delta  \int_{t_0}^1 \|\psi(v,D_v)M^{k} f\|_{(2,0)}^2  dt+\frac{1}{\delta}\bigg[\frac{\epsilon  }{(2k-1)^3}  C_*^{ 2k-2} (2k)!\bigg]^2.
\end{align*}
 Finally by Proposition \ref{prop:com},
\begin{eqnarray*}
\begin{aligned}
	&\int_{t_0}^1\big|\left(M^k\Gamma(f, f)-\Gamma(f, M^kf), M^k f\right)_{(2,0)}\big|dt+\int_{t_0}^1\big|\left([M^k,\,\mathcal{L}]\, f,\, M^k f\right)_{(2,0)}\big| dt \\
&	\leq \inner{ \delta +\epsilon C_2}\bigg[ \Big(\sup_{t_0\leq t\leq 1} \norm{M^{k} f }_{(2,0)}\Big)^2+  \int_{t_0}^1\norm{ \psi(v,D_v)  M^{k}   f}_{(2,0)}^2dt   \bigg]
	  +  C_\delta \bigg[ \Big(\epsilon+C_*^{-1}\Big) \frac{\epsilon}{(2k+1)^3}  C_*^{ 2k} (2k)!  \bigg]^2.
	  \end{aligned}
\end{eqnarray*}
  We combine the above estimates with \eqref{commut-1} to conclude
  \begin{multline*}
  	  \frac{1}{2} \|M^k f(t)\|^2_{(2,0)}+\frac{1}{C_1} \int_{t_0}^t  \norm{\psi(v,D_v) M^k f}_{(2,0)}^2 ds \\
  	 \leq \int_{t_0}^t \norm{ M^k f}_{(2,0)}^2 ds
  	  +\inner{ \delta +\epsilon C_2}\bigg[ \Big(\sup_{t_0\leq t\leq 1} \norm{M^{k} f }_{(2,0)}\Big)^2+  \int_{t_0}^1\norm{ \psi(v,D_v)  M^{k}   f}_{(2,0)}^2dt   \bigg]  \\+  C_\delta    \bigg[ \Big(\epsilon+C_*^{-1}\Big) \frac{\epsilon}{(2k+1)^3}  C_*^{ 2k} (2k)!  \bigg]^2,
  \end{multline*}
which with Gronwall's  inequality yields, for any $\delta>0,$
 \begin{eqnarray*}
 \begin{aligned}
  	   &\sup_{t_0\leq t\leq 1} \norm{M^{k} f }_{(2,0)} +\frac{1}{C_1} \bigg( \int_{t_0}^1\norm{ \psi(v,D_v)  M^{k}   f}_{(2,0)}^2dt \bigg)^{1/2} \\
  	& \leq 8 \inner{ \delta +\epsilon C_2}^{1\over2}\bigg[  \sup_{t_0\leq t\leq 1} \norm{M^{k} f }_{(2,0)}+  \bigg(\int_{t_0}^1\norm{ \psi(v,D_v)  M^{k}   f}_{(2,0)}^2dt\bigg)^{1\over2}   \bigg]
	  +  8C_\delta \Big(\epsilon+C_*^{-1}\Big) \frac{\epsilon C_*^{ 2k} (2k)!}{(2k+1)^3}  .
	  \end{aligned}
  \end{eqnarray*}
 This implies, choosing $\delta$ small sufficiently and using the smallness assumption on $\epsilon,$
\begin{align*}
		  \sup_{t_0\leq t\leq 1} \norm{M^{k} f }_{(2,0)} +  \bigg(\int_{t_0}^1\norm{ \psi(v,D_v)  M^{k}   f}_{(2,0)}^2dt\bigg)^{1\over 2}&\leq C  \Big(\epsilon+C_*^{-1}\Big) \frac{\epsilon  }{(2k+1)^3}  C_*^{ 2k} (2k)!
\end{align*}
for some constant $C$ depending only on the constants $C_1$ and $C_2$ given in Propositions \ref{prop: tricoe} and \ref{prop:com} but independent of $\epsilon$ and $C_*$.  Thus we conclude the validity of \eqref{+ass:ind} for $j=k$,
provided $C_*>2C$ and $\epsilon$ is small enough such that $\epsilon C\leq 1/2$.   We complete the proof of  Theorem \ref{thm:key}.
\end{proof}

  \section{Trilinear and coercivity estimates}

  In this part we will prove the quantitative  estimates in Propositions \ref{prop: tricoe}.  The proof is quite lengthy, and we proceed through the following subsections.

  \subsection{Analysis for the linear collision operator}
  This part is devoted to deriving the representation of the Landau collision operator  in terms of differential operators involving the Laplacian $\Delta_v$ and the Laplace-Beltrami operator $(v\wedge\partial_v)^2$ on the unite sphere $\mathbb S^2$. This enables to complete the proof of     Proposition \ref{prop: tricoe}.

Recall  $g*h$ stands for the convolution with respect to  $v$ only.
Let  $a_{i,j}$ be given in \eqref{coe} and  denote
\begin{equation}\label{rho}
\bar a_{i,j}(f)=a_{i,j}*f =\int_{\mathbb R^3}a_{i,j}(v-v_*)f(v_*)dv_*.
\end{equation}
For    a given function $g$, define
\begin{equation}\label{ag}
a_g=\abs v^\gamma *(\sqrt \mu g), \   \boldsymbol{A}_g=(a_{1,g},a_{2,g},a_{3,g})=\abs v^\gamma *(\sqrt\mu\partial_v g) \textrm{ and }
\boldsymbol{B}_{g}=(b_{1,g},b_{2,g},b_{3,g})=\abs v^\gamma *(v\sqrt \mu g),
\end{equation}
and moreover define $M_{i,j,g}, \rho_{i,j,g}, \lambda_{i,j,g}, 1\leq i,j\leq 3,$ as below.
\begin{equation}\label{M}
\left\{
\begin{aligned}
&M_{i,j,g}=\abs v^\gamma*\inner{(\delta_{i,j}\abs{v}^2-v_iv_j)\sqrt\mu g},
\\&
\rho_{i,j,g}(v)=\int_{\mathbb R^3}
\abs{v-v_*}^{\gamma}\big(\sqrt\mu\partial_{v_i}g\big)(v_*) \inner{-2 v_j(v_*)_j+(v_*)_j(v_*)_j}dv_*
\\&\qquad\qquad\qquad
+\int_{\mathbb R^3}\abs{v-v_*}^{\gamma}
\big(\sqrt\mu\partial_{v_j}g\big)(v_*)\inner{ v_i(v_*)_j+v_j(v_*)_i-(v_*)_i(v_*)_j}dv_*,
\\&
\lambda_{i,j,g}(v)
=\int_{\mathbb R^3}\abs{v-v_*}^{\gamma}\big(\sqrt\mu v_j g\big)(v_*)\inner{v_i(v_*)_j-v_j(v_*)_i}dv_*,\end{aligned}
\right.
\end{equation}
where and below  $(v_*)_i$  stands for the $i^{th}$ entry of the vector $v_*\in\mathbb R^3.$

 \begin{proposition}\label{proprep}
Let $\Gamma(g,h)$ be the quadratic operator defined by \eqref{gm}. Then, with the nations given by \eqref{ag}-\eqref{M},
\begin{equation*}\label{Non}
\Gamma(g,h)=\sum_{1\leq j\leq 6}L_j(g,h),
\end{equation*}
where
\begin{equation}\label{defl}
\left\{
\begin{aligned}
L_1(g,h)
&=\frac{1}{2}\inner{v\wedge\partial_v}\cdot a_g(v)\inner{v\wedge \partial_v}h+\sum_{1\leq i,j\leq3}\partial_{v_i}\Big(M_{i,j,g}  \partial_{v_j}h\Big)\\
&\quad +\frac{1}{2}\inner{\partial_v\wedge\boldsymbol{B}_g}
\cdot \inner{v\wedge \partial_v}h-\frac{1}{2}(v\wedge\partial_v)\cdot(\boldsymbol{B}_{g}\wedge\partial_v)h,
\\
L_2(g,h)
&=\frac{1}{4}\Big[(v\wedge\partial_v)\cdot(\boldsymbol{B}_{g}\wedge v)h+ (\boldsymbol{B}_{g}\wedge v) \cdot (v\wedge\partial_v)h\Big]
-\frac{1}{2}\sum_{1\leq i,j\leq3}
\Big[\partial_{v_i }(M_{i,j,g}v_j h)+v_iM_{i,j,g}\partial_{v_j}h\Big]
\\
L_3(g,h)
&=\frac{1}{4}\sum_{1\leq i,j\leq3}v_i M_{i,j,g}v_j  h,
\end{aligned}
\right.
\end{equation}
and
\begin{eqnarray*}
\left\{
\begin{aligned}
L_4(g,h)
&=\frac{1}{2}\inner{v\wedge\partial_v}\cdot\inner{\boldsymbol{A}_g\wedge v} h
-\sum_{\stackrel{1\leq i,j\leq3}{i\neq j}}\partial_{v_i}\big(\rho_{i,j,g} h\big),
\\
L_5(g,h)
&=-\frac{1}{4}\inner{v\wedge \partial_v}\cdot\inner{\boldsymbol{B}_g\wedge v} h
+\frac{1}{2}\sum_{\stackrel{1\leq i,j\leq3}{i\neq j}}\partial_{v_i} \big(\lambda_{i,j,g} h\big)
+\frac{1}{2}\sum_{\stackrel{1\leq i,j\leq3}{i\neq j}}v_i \rho_{i,j,g} h,
\\
L_6(g,h)
&=-\frac{1}{4}\sum_{\stackrel{1\leq i,j\leq3}{i\neq j}}v_i \lambda_{i,j,g} h.
\end{aligned}
\right.
\end{eqnarray*}
Recall $	\xi\wedge\zeta$  stands for the  cross product defined by \eqref{crosec}.
\end{proposition}

To prove Proposition  \ref{proprep}   we first list the representations of $L_j(g,h),1\leq j\leq 6.$

\begin{lemma}\label{Non-1}
Let $L_{j}(g,h), 1\leq j\leq 3,$ be the bilinear operators defined in Proposition \ref{proprep}. Then we have
\begin{eqnarray*}
\begin{aligned}
L_1(g,h)
=&\sum_{1\leq i,j\leq3}\partial_{v_i}(\bar a_{i,j}\inner{\sqrt{\mu}g}\partial_{v_j}h),
\\
L_2(g,h)
=&-\frac{1}{2}\sum_{1\leq i,j\leq3}\partial_{v_i}\Big(\bar a_{i,j}\inner{\sqrt{\mu}g}v_j h\Big)
-\frac{1}{2}\sum_{1\leq i,j\leq3}v_i\bar a_{i,j}\inner{\sqrt{\mu}g}\partial_{v_j} h,
\\
L_3(g,h)
=&\frac{1}{4}\sum_{1\leq i,j\leq3}v_i\bar a_{i,j}\inner{\sqrt{\mu}g}v_j h.
\end{aligned}
\end{eqnarray*}
\end{lemma}

\begin{proof}
  Recall $a_{i,j}$ is given in  \eqref{coe}.  Using the fact that
\begin{equation}\label{aiij}
a_{i,i}(z)=\abs{z}^\gamma \sum_{\stackrel{1\leq j\leq 3}{j\neq i}}z_j^2,  \qquad
a_{i,j}(z)=-\abs{z}^\gamma z_iz_j \ \  {\rm for } \ \ i\neq j,
\end{equation}
 we compute, observing the notation  in \eqref{rho},
\begin{align*}
 \sum_{1\leq i,j\leq3}\partial_{v_i}(\bar a_{i,j}\inner{\sqrt{\mu}g}\partial_{v_j}h(v))
&=\sum_{\stackrel{1\leq i,j\leq3}{i\neq j}}\partial_{v_i}\int_{\mathbb R^3}\abs{v-v_*}^{\gamma}\abs{v_j-(v_*)_j}^{2}
\inner{\sqrt{\mu}g}(v_*)\partial_{v_i}h(v)dv_*
\\&\quad -\sum_{\stackrel{1\leq i,j\leq3}{i\neq j}}\partial_{v_i}\int_{\mathbb R^3}\abs{v-v_*}^{\gamma}\inner{v_i-(v_*)_i}\inner{v_j-(v_*)_j}
\inner{\sqrt{\mu}g}(v_*)\partial_{v_j}h(v)dv_*
\\&
 :=  I_1-I_2. 	
\end{align*}
Moreover, using the notations in \eqref{ag},
\begin{eqnarray*}
\begin{aligned}
I_1
&=\sum_{\stackrel{1\leq i,j\leq3}{i\neq j}}\partial_{v_i}\int_{\mathbb R^3}\abs{v-v_*}^{\gamma} \Big(v_j^2-2v_j(v_*)_j+(v_*)_j^{2}\Big)
\inner{\sqrt{\mu}g}(v_*)\partial_{v_i}h(v)dv_*
\\&
=\sum_{\stackrel{1\leq i,j\leq3}{i\neq j}}v_j\partial_{v_i}(a_g(v)v_j\partial_{v_i}h)
-2\sum_{\stackrel{1\leq i,j\leq3}{i\neq j}}v_j\partial_{v_i}(b_{j,g}(v)\partial_{v_i}h)
\\&\quad
+\sum_{\stackrel{1\leq i,j\leq3}{i\neq j}}\partial_{v_i}\int_{\mathbb R^3}\abs{v-v_*}^{\gamma} (v_*)_j^{2}
\inner{\sqrt{\mu}g}(v_*)\partial_{v_i}h(v)dv_*
\\&
 :=  I_{1,1}+I_{1,2}+I_{1,3}
\end{aligned}
\end{eqnarray*}
and
\begin{eqnarray*}
\begin{aligned}
I_2
&=\sum_{\stackrel{1\leq i,j\leq3}{i\neq j}}\partial_{v_i}\int_{\mathbb R^3}\abs{v-v_*}^{\gamma}\inner{v_iv_j -\Big(v_i(v_*)_j+v_j(v_*)_i\Big)+ (v_*)_i (v_*)_j}
\inner{\sqrt{\mu}g}(v_*)\partial_{v_j}h(v) dv_*
\\&=
\sum_{\stackrel{1\leq i,j\leq3}{i\neq j}}v_j\partial_{v_i}(a_g(v)v_i\partial_{v_j}h)
-\sum_{\stackrel{1\leq i,j\leq3}{i\neq j}}\Big(\partial_{v_i}(b_{j,g}(v)v_i\partial_{v_j}h)
+v_j\partial_{v_i}(b_{i,g}(v)\partial_{v_j}h)\Big)
\\&\quad
+\sum_{\stackrel{1\leq i,j\leq3}{i\neq j}}\partial_{v_i}\int_{\mathbb R^3}\abs{v-v_*}^{\gamma}(v_*)_i(v_*)_j
\inner{\sqrt{\mu}g}(v_*)\partial_{v_j}h(v)dv_*
\\&
 :=  I_{2,1}+I_{2,2}+I_{2,3}.	
\end{aligned}
\end{eqnarray*}
Direct verification shows
\begin{eqnarray*}
\begin{aligned}
I_{1,1}-I_{2,1}=\sum_{\stackrel{1\leq i,j\leq3}{i\neq j}}\inner{v_j\partial_{v_i}(a_g (v)v_j\partial_{v_i}h)
-v_j\partial_{v_i}(a_g(v)v_i\partial_{v_j}h)}
=\frac{1}{2}\inner{v\wedge\partial_v}\cdot a_g(v)\inner{v\wedge \partial_v}h.
\end{aligned}
\end{eqnarray*}
Similarly,
\begin{eqnarray*}
\begin{aligned}
&I_{1,2}-I_{2,2}
=\sum_{\stackrel{1\leq i,j\leq3}{i\neq j}}\partial_{v_i}(b_{j,g}(v)v_i\partial_{v_j}h)
+\sum_{\stackrel{1\leq i,j\leq3}{i\neq j}}v_j\partial_{v_i}(b_{i,g} (v) \partial_{v_j}h)
-2\sum_{\stackrel{1\leq i,j\leq3}{i\neq j}}v_j\partial_{v_i}(b_{j,g} (v)\partial_{v_i}h)\\
&
=\sum_{\stackrel{1\leq i,j\leq3}{i\neq j}}\Big(\partial_{v_i}(b_{j,g}(v)v_i\partial_{v_j}h)-\partial_{v_i}(b_{j,g} (v)v_j\partial_{v_i}h)\Big)
+\sum_{\stackrel{1\leq i,j\leq3}{i\neq j}}\Big(v_j\partial_{v_i}(b_{i,g} (v) \partial_{v_j}h)
- v_j\partial_{v_i}(b_{j,g} (v)\partial_{v_i}h)\Big)
\\&
=\frac{1}{2}\inner{\partial_v\wedge\boldsymbol{B}_g}
\cdot \inner{v\wedge \partial_v}h-\frac{1}{2}(v\wedge\partial_v)\cdot(\boldsymbol{B}_{g}\wedge\partial_v)h.
\end{aligned}
\end{eqnarray*}
Finally, using the fact that
\begin{eqnarray*}
\delta_{i,i}\abs{v_*}^2-(v_*)_i(v_*)_i=\sum_{\stackrel{1\leq j\leq 3}{j\neq i}} (v_*)_j ^2, \qquad
\delta_{i,j}\abs{v_*}^2-(v_*)_i(v_*)_j=-(v_*)_i(v_*)_j  \ \textrm{ for } \   i\neq j,
\end{eqnarray*}
yields
\begin{eqnarray*}
\begin{aligned}
I_{1,3}-I_{2,3}
&=\sum_{\stackrel{1\leq i,j\leq3}{i\neq j}}\partial_{v_i} \int_{\mathbb R^3}\abs{v-v_*}^{\gamma}\inner{(v_*)_j^{2}\partial_{v_i}h(v)-(v_*)_i(v_*)_j\partial_{v_j}h(v)}
\inner{\sqrt{\mu}g}(v_*) dv_*
\\&
=\sum_{1\leq i,j\leq3}\partial_{v_i} \int_{\mathbb R^3}\abs{v-v_*}^{\gamma}
\inner{\delta_{i,j}\abs{v_*}^2-(v_*)_i(v_*)_j}(\sqrt{\mu}g)(v_*)\partial_{v_j}h(v)dv_* \\
&
=\sum_{1\leq i,j\leq3}\partial_{v_i}\Big(M_{i,j,g}  \partial_{v_j}h\Big).
\end{aligned}
\end{eqnarray*}
Combining the above equalities we conclude
\begin{equation}\label{l1est}
	\begin{aligned}
 \sum_{1\leq i,j\leq3}\partial_{v_i}(\bar a_{i,j}\inner{\sqrt{\mu}g}\partial_{v_j}h(v))
&=\frac{1}{2}\inner{v\wedge\partial_v}\cdot a_g(v)\inner{v\wedge \partial_v}h+\sum_{1\leq i,j\leq3}\partial_{v_i}\Big(M_{i,j,g}  \partial_{v_j}h\Big)\\
&\quad +\frac{1}{2}\inner{\partial_v\wedge\boldsymbol{B}_g}
\cdot \inner{v\wedge \partial_v}h-\frac{1}{2}(v\wedge\partial_v)\cdot(\boldsymbol{B}_{g}\wedge\partial_v)h\\
&=L_1(g,h),
\end{aligned}
\end{equation}
the last line using the definition \eqref{defl}. We have proven the first assertion in Lemma \ref{Non-1}.

Similarly,  we  replace   the differential operator $\partial_{v_i}$  in \eqref{l1est}  by $v_i$, or replace $\partial_{v_j}h$   by $v_jh$, and observe $v\wedge v=\partial_v\wedge\partial_v=0$; this yields
\begin{eqnarray*}
	-\frac{1}{2}\sum_{1\leq i,j\leq3}\partial_{v_i}(\bar a_{i,j}\inner{\sqrt{\mu}g}v_j h)=-\frac{1}{2}\sum_{1\leq i,j\leq3}\partial_{v_i}\Big(M_{i,j,g}   v_j h\Big)+\frac{1}{4}(v\wedge\partial_v)\cdot(\boldsymbol{B}_{g}\wedge v)h
\end{eqnarray*}
and
\begin{eqnarray*}
	-\frac{1}{2}\sum_{1\leq i,j\leq3}v_i\bar a_{i,j}\inner{\sqrt{\mu}g}\partial_{v_j} h=-\frac12  \sum_{1\leq i,j\leq3} v_i  M_{i,j,g}  \partial_{v_j}h  -\frac{1}{4}\inner{v\wedge\boldsymbol{B}_g}
\cdot \inner{v\wedge \partial_v}h ,
\end{eqnarray*}
and thus, in view of \eqref{defl} and the fact that $\xi\wedge\zeta=-\zeta\wedge\xi$,
\begin{multline*}
 	-\frac{1}{2}\sum_{1\leq i,j\leq3}\partial_{v_i}(\bar a_{i,j}\inner{\sqrt{\mu}g}v_j h)
-\frac{1}{2}\sum_{1\leq i,j\leq3}v_i\bar a_{i,j}\inner{\sqrt{\mu}g}\partial_{v_j} h\\
   =\frac{1}{4}\Big((v\wedge\partial_v)\cdot(\boldsymbol{B}_{g}\wedge v)h+ (\boldsymbol{B}_{g}\wedge v) \cdot (v\wedge\partial_v)h\Big)-\frac12\sum_{1\leq i,j\leq3}
\inner{\partial_{v_i }(M_{i,j,g}v_j h)+v_iM_{i,j,g}\partial_{v_j}h}=L_2(g,h)
\end{multline*}
Moreover, we  replace    the differential operator $\partial_{v_i}$ and the function  $\partial_{v_j}h$  in \eqref{l1est}  by $v_i$ and $v_jh$, respectively, to get
\begin{eqnarray*}
	\frac{1}{4}\sum_{1\leq i,j\leq3}v_i\bar a_{i,j}\inner{\sqrt{\mu}g}v_j h=\frac{1}{4}\sum_{1\leq i,j\leq3}v_i M_{i,j,g}v_j  h=L_{3}(g,h),
\end{eqnarray*}
the last equalities using again \eqref{defl}.  Combining the above   equalities with \eqref{l1est} we complete the proof of Lemma \ref{Non-1}.
\end{proof}

\begin{lemma}\label{Non-2} Let $L_{j}(g,h), 4\leq j\leq 6,$ be the bilinear operators defined in Proposition \ref{proprep}.
It holds that
\begin{eqnarray*}
\begin{aligned}
L_4(g,h)
&=-\sum_{1\leq i,j\leq3}\partial_{v_i}\Big(\bar a_{i,j}\big(\sqrt{\mu}\partial_{v_j}g\big)h\Big),
\\
L_5(g,h)
&=\frac{1}{2} \sum_{1\leq i,j\leq3}\partial_{v_i}\Big(\bar a_{i,j}\big(\sqrt{\mu}v_j g\big)h\Big)
+\frac{1}{2}\sum_{1\leq i,j\leq3}v_i\bar a_{i,j}\inner{\sqrt{\mu}\partial_{v_j}g}h,
\\
L_6(g,h)
&=-\frac{1}{4}\sum_{1\leq i,j\leq3}v_i\bar a_{i,j}\inner{\sqrt{\mu}v_j g} h.
\end{aligned}
\end{eqnarray*}
\end{lemma}

\begin{proof}
 Using the notation in \eqref{rho} as well as \eqref{aiij},   we may write
\begin{eqnarray*}
\begin{aligned}
-\sum_{1\leq i,j\leq3}\partial_{v_i}\Big(\bar a_{i,j}\big(\sqrt{\mu}\partial_{v_j}g\big)h\Big)&=-\sum_{\stackrel{1\leq i,j\leq3}{i\neq j}}\partial_{v_i}\int_{\mathbb R^3}\abs{v-v_*}^{\gamma}\abs{v_j-(v_*)_j}^{2}
(\sqrt\mu\partial_{v_i}g)(v_*)h(v)dv_*
\\&\quad
+\sum_{\stackrel{1\leq i,j\leq3}{i\neq j}}\partial_{v_i}\int_{\mathbb R^3}
\abs{v-v_*}^{\gamma}\inner{v_i-(v_*)_i}\inner{v_j-(v_*)_j}(\sqrt\mu\partial_{v_j}g)(v_*)h(v)dv_*
\\&
 := K_1+K_2,
\end{aligned}
\end{eqnarray*}
and morover
\begin{eqnarray*}
\begin{aligned}
K_1&
=-\sum_{\stackrel{1\leq i,j\leq3}{i\neq j}}\partial_{v_i}\int_{\mathbb R^3}\abs{v-v_*}^{\gamma}
\inner{(v_*)_j(v_*)_j-2v_j(v_*)_j}(\sqrt\mu\partial_{v_i}g)(v_*)h(v) dv_*
\\&\quad
-\sum_{\stackrel{1\leq i,j\leq3}{i\neq j}}v_j\partial_{v_i}\int_{\mathbb R^3}\abs{v-v_*}^{\gamma}v_j
(\sqrt\mu\partial_{v_i}g)(v_*)h(v) dv_*
\end{aligned}
\end{eqnarray*}
and
\begin{eqnarray*}
\begin{aligned}
K_2&
=\sum_{\stackrel{1\leq i,j\leq3}{i\neq j}}\partial_{v_i}\int_{\mathbb R^3}\abs{v-v_*}^{\gamma}
\inner{ (v_*)_i (v_*)_j-v_i(v_*)_j-(v_*)_i v_j}(\sqrt\mu\partial_{v_j}g)(v_*) h(v)dv_*
\\&\quad
+\sum_{\stackrel{1\leq i,j\leq3}{i\neq j}}v_j\partial_{v_i}\int_{\mathbb R^3}\abs{v-v_*}^{\gamma}v_i
(\sqrt\mu\partial_{v_j}g)(v_*) h(v)dv_*.
\end{aligned}
\end{eqnarray*}
Using the nations in \eqref{M} and \eqref{ag} gives
\begin{multline*}
	\sum_{\stackrel{1\leq i,j\leq3}{i\neq j}}\partial_{v_i}\int_{\mathbb R^3}\abs{v-v_*}^{\gamma}
\inner{ (v_*)_i (v_*)_j-v_i(v_*)_j-(v_*)_i v_j}(\sqrt\mu\partial_{v_j}g)(v_*) h(v)dv_*\\
-\sum_{\stackrel{1\leq i,j\leq3}{i\neq j}}\partial_{v_i}\int_{\mathbb R^3}\abs{v-v_*}^{\gamma}
\inner{(v_*)_j(v_*)_j-2v_j(v_*)_j}(\sqrt\mu\partial_{v_i}g)(v_*)h(v) dv_*=-\sum_{\stackrel{1\leq i,j\leq3}{i\neq j}}\partial_{v_i}\big(\rho_{i,j,g} h\big)
\end{multline*}
and
\begin{multline*}
	\sum_{\stackrel{1\leq i,j\leq3}{i\neq j}}v_j\partial_{v_i}\int_{\mathbb R^3}\abs{v-v_*}^{\gamma}v_i
(\sqrt\mu\partial_{v_j}g)(v_*) h(v)dv_*\\ -\sum_{\stackrel{1\leq i,j\leq3}{i\neq j}}v_j\partial_{v_i}\int_{\mathbb R^3}\abs{v-v_*}^{\gamma}v_j
(\sqrt\mu\partial_{v_i}g)(v_*)h(v) dv_*
= \frac{1}{2}\inner{v\wedge\partial_v}\cdot\inner{\boldsymbol{A}_g\wedge v} h.
\end{multline*}
Consequently,  we combine the above equalities to conclude
\begin{eqnarray*}
	-\sum_{1\leq i,j\leq3}\partial_{v_i}\Big(\bar a_{i,j}\big(\sqrt{\mu}\partial_{v_j}g\big)h\Big)=K_1+K_2=\frac{1}{2}\inner{v\wedge\partial_v}\cdot\inner{\boldsymbol{A}_g\wedge v} h
-\sum_{\stackrel{1\leq i,j\leq3}{i\neq j}}\partial_{v_i}\big(\rho_{i,j,g} h\big)=L_4(g,h),
\end{eqnarray*}
the last inequality using the definition of $L_4(g,h)$ given in Proposition \ref{proprep}.
Similarly, we can verify that
\begin{multline*}
\frac{1}{2} \sum_{1\leq i,j\leq3}\partial_{v_i}\Big(\bar a_{i,j}\big(\sqrt{\mu}v_j g\big)h\Big)
+\frac{1}{2}\sum_{1\leq i,j\leq3}v_i\bar a_{i,j}\inner{\sqrt{\mu}\partial_{v_j}g}h\\
=-\frac{1}{4}\inner{v\wedge\partial_v}\cdot\inner{\boldsymbol{B}_g\wedge v} h
+\frac12\sum_{\stackrel{1\leq i,j\leq3}{i\neq j}}\partial_{v_i}\big(\lambda_{i,j,g} h\big)+\frac12\sum_{\stackrel{1\leq i,j\leq3}{i\neq j}} v_i \rho_{i,j,g} h =L_5(g,h)
\end{multline*}
and
\begin{eqnarray*}
-\frac{1}{4}\sum_{1\leq i,j\leq3}v_i\bar a_{i,j}\inner{\sqrt{\mu}v_j g}h=-\frac14\sum_{\stackrel{1\leq i,j\leq3}{i\neq j}} v_i  \lambda_{i,j,g} h =L_6(g,h),
\end{eqnarray*}
with $L_5(g,h), L_6(g,h)$ defined in Proposition \ref{proprep}. The proof of Lemma \ref{Non-2} is completed.
\end{proof}

\begin{proof}
	[Proof of Proposition \ref{proprep}]
	Recalling $\Gamma(g,h)$ is defined by \eqref{gm} and in view of  the representation \eqref{colli},  we write
\begin{align*}
 \Gamma(g, h) &=\mu^{-\frac 12 }\sum_{1\leq i,j\leq3}\partial_{v_i} \int_{\mathbb R^3}a_{i,j}(v-v_*)\big[(\sqrt{\mu} g)(v_*)\partial_{v_j}
(\sqrt{\mu} h)(v)-(\sqrt{\mu} h)(v)\big(\partial_{v_j}(\sqrt{\mu} g)\big)(v_*)\big]dv_*\\
&=\mu^{-1/2}\sum_{1\leq i,j\leq3}\partial_{v_i}\int_{\mathbb R^3} a_{i,j}(v-v_*)
(\sqrt{\mu}g)(v_*) \sqrt{\mu(v)}\inner{\partial_{v_j}h(v)-\frac{v_j}{2}h(v)}dv_*
\\&\quad
-\mu^{-1/2} \sum_{1\leq i,j\leq
3}\partial_{v_i}\int_{\mathbb R^3} a_{i,j}(v-v_*)(\sqrt{\mu}h)(v)\Big[
\inner{\sqrt{\mu}\partial_{v_j}g}(v_*)-\inner{v_j\sqrt{\mu}g/2}(v_*)\Big] dv_*. 	
\end{align*}
 Then using the notation in \eqref{rho}
we can split the terms on the right hand side as below.
\begin{equation*}\label{reoflan}
\begin{aligned}
\Gamma(g,h)
=&\sum_{1\leq i,j\leq3}\partial_{v_i}\Big(\bar a_{i,j}\inner{\sqrt{\mu}g}\partial_{v_j} h\Big)
-\frac{1}{2}\sum_{1\leq i,j\leq3}\partial_{v_i}\Big(\bar a_{i,j}\inner{\sqrt{\mu}g}v_j h\Big)
\\&
-\frac{1}{2}\sum_{1\leq i,j\leq3}v_i\bar a_{i,j}\inner{\sqrt{\mu}g}\partial_{v_j} h
+\frac{1}{4}\sum_{1\leq i,j\leq3}v_i\bar a_{i,j}\inner{\sqrt{\mu}g}v_j h
\\&
-\sum_{1\leq i,j\leq3}\partial_{v_i}\Big(\bar a_{i,j}\big(\sqrt{\mu}\partial_{v_j}g\big) h\Big)
+\frac{1}{2}\sum_{1\leq i,j\leq3}\partial_{v_i}\Big(\bar a_{i,j}\big(\sqrt{\mu}v_j g\big) h\Big)\\&
+\frac{1}{2}\sum_{1\leq i,j\leq3}v_i\bar a_{i,j}\big(\sqrt{\mu}\partial_{v_j}g\big) h
-\frac{1}{4} \sum_{1\leq i,j\leq3}v_i\bar a_{i,j}\big(\sqrt{\mu}v_j g\big) h,
\end{aligned}
\end{equation*}
which with   Lemmas \ref{Non-1} and \ref{Non-2} yields the assertion in Proposition \ref{proprep}. The proof is thus completed.
 \end{proof}

\subsection{Proof of Proposition \ref{prop: tricoe}: trilinear and coercivity estimates}
For simplicity of notations, we denote by  $C$  a generic constant.     We first derive the trilinear estimate \eqref{trili}, and   it suffices to prove that, in view of Proposition \ref{proprep},
\begin{equation}\label{uptriestm}
  \sum_{i=1}^{i=6}\big|\inner{   L_i(g,h),\  \omega }_{L_v^2}\big| \leq C\norm{g}_{L_v^2}\norm{\psi(v,D_v)h}_{L_v^2}\norm{\psi(v,D_v)\omega}_{L_v^2},
\end{equation}
where $L_j(g,h), 1\leq j\leq 6,$ are defined in Proposition \ref{proprep}.   By the representation  of  $L_1$  in \eqref{defl}, it follows     that
\begin{eqnarray*}
\begin{aligned}
&\big|\inner{   L_1(g,h),\  \omega }_{L_v^2}\big|=\Big|\int_{\mathbb R_v^3}   L_1(g,h)  \omega dv\Big| \\
&\leq \frac{1}{2}\norm{\comi v^{-{\gamma\over 2}} a_g(v)\inner{v\wedge\partial_v}h}_{L_v^2} \norm{\comi v^{\gamma\over 2}\inner{v\wedge\partial_v} \omega }_{L_v^2}  +\sum_{1\leq i,j\leq3}\norm{\comi v^{-{\gamma\over 2}}  M_{i,j,g}(v) \partial_{v_j}h}_{L_v^2} \norm{\comi v^{\gamma\over2} \partial_{v_i}\omega }_{L_v^2}
\\&\quad
+\frac{1}{2}  \norm{\comi v^{ \gamma\over2} \inner{v\wedge\partial_v}h}_{L_v^2}  \norm{\comi v^{-{\gamma\over 2}} \big({\boldsymbol{B}_g}(v)\wedge \partial_v\big)\omega}_{L_v^2}
 + \frac{1}{2}  \norm{\comi v^{ \gamma\over2} \inner{v\wedge\partial_v}\omega}_{L_v^2}  \norm{\comi v^{-{\gamma\over 2}} \big({\boldsymbol{B}_g}(v)\wedge \partial_v\big)h}_{L_v^2} .
\end{aligned}
\end{eqnarray*}
In view of \eqref{ag}-\eqref{M}, we have
\begin{eqnarray*}
	\forall\    v\in\mathbb R^3,\quad |a_g(v)|+\sum_{1\leq i,j\leq 3}|M_{i,j,g}(v)|+|{\boldsymbol{B}_g}(v)|\leq C \comi v^\gamma \norm{g}_{L_v^2}.
\end{eqnarray*}
This with \eqref{ma+0} yields
\begin{equation}\label{bouforl1}
\big|\inner{   L_1(g,h),\  \omega }_{L_v^2}\big|   \leq 	C \norm{g}_{L_v^2}\norm{\psi(v,D_v)h}_{L_v^2}\norm{\psi(v,D_v)\omega}_{L_v^2}.
\end{equation}
In view the representations of $L_i(g,h),1\leq i\leq 6,$ in Proposition \ref{proprep},   the above estimate \eqref{bouforl1} still holds true with $L_1(g,h)$ replaced by $L_i(g,h)$ with $i=2,3$ or $6$.

To complete the proof of  \eqref{uptriestm},
it remains to estimate $\big|\inner{   L_i(g,h),\  \omega }_{L_v^2}\big|$ with $i=4$ or $5.$
To do so we combine \eqref{bouforl1} with Lemma \ref{Non-1} to conclude
\begin{eqnarray*}
	\Big| \sum_{1\leq i,j\leq3}\inner{ \partial_{v_i}\big (\bar a_{i,j}\inner{\sqrt{\mu}g}\partial_{v_j}h\big ),\  \omega }_{L_v^2}\Big|	 =\big|\inner{   L_1(g,h),\  \omega }_{L_v^2}\big| \leq 	C \norm{g}_{L_v^2}\norm{\psi(v,D_v)h}_{L_v^2}\norm{\psi(v,D_v)\omega}_{L_v^2}.
\end{eqnarray*}
Similarly, replacing $\sqrt \mu  g$ and $\partial_{v_j}h$  above by   $v_j\sqrt \mu  g$ and $ h$, respectively,
\begin{equation}\label{omete}
	\Big| \sum_{1\leq i,j\leq3}\inner{ \partial_{v_i}\big (\bar a_{i,j}\big(v_j\sqrt{\mu}g\big)  h\big ),\  \omega }_{L_v^2}\Big|	 \leq 	C \norm{g}_{L_v^2}\norm{\psi(v,D_v)h}_{L_v^2}\norm{\psi(v,D_v)\omega}_{L_v^2}.
\end{equation}
By Lemma \ref{Non-2},
\begin{eqnarray*}
\big|\inner{   L_4(g,h),\  \omega }_{L_v^2}\big|= \Big|\sum_{1\leq i,j\leq3}\inner{\partial_{v_i}\big(\bar a_{i,j}\big(\sqrt{\mu}\partial_{v_j}g\big)h\big), \ \omega}_{L_v^2}\Big|.
\end{eqnarray*}
This with the fact that,  recalling the notation in \eqref{rho},
\begin{align*}
	\bar a_{i,j}\big(\sqrt{\mu}\partial_{v_j}g\big) = a_{i,j} * \Big( \partial_{v_j}(\sqrt{\mu}g) -g \partial_{v_j} \sqrt{\mu} \Big)=\big(\partial_{v_j}a_{i,j}\big)* \big(\sqrt{\mu}g\big)+\frac12 a_{i,j}* \big( v_j \sqrt{\mu} g\big),
\end{align*}
yields
\begin{align*}
	&\big|\inner{   L_4(g,h),\  \omega }_{L_v^2}\big| \leq \sum_{1\leq i,j\leq3}\Big|\inner{\partial_{v_i}\big(\overline{\partial_{v_j} a_{i,j}}\big(\sqrt{\mu}g\big)h\big), \ \omega}_{L_v^2}\Big|+\frac12 \Big|\sum_{1\leq i,j\leq3}\inner{\partial_{v_i}\big(\bar a_{i,j}\big(v_j\sqrt{\mu}  g\big)h\big), \ \omega}_{L_v^2}\Big|\\
	&\leq \sum_{1\leq i,j\leq3} \norm{\comi v^{-{\gamma\over2}}  \overline{\partial_{v_j} a_{i,j}}\big(\sqrt{\mu}g\big)h}_{L_v^2}\norm{\comi v^{\gamma\over2}\partial_{v_i}\omega}_{L_v^2} +\frac12\Big|\sum_{1\leq i,j\leq3}\inner{\partial_{v_i}\big(\bar a_{i,j}\big(v_j\sqrt{\mu}  g\big)h\big), \ \omega}_{L_v^2}\Big|\\
	&  \leq 	C \norm{g}_{L_v^2}\norm{\psi(v,D_v)h}_{L_v^2}\norm{\psi(v,D_v)\omega}_{L_v^2},
\end{align*}
 the last inequality using \eqref{omete} and the fact that
 \begin{eqnarray*}
 	\forall\  v\in\mathbb R^3, \quad \big | \overline{\partial_{v_j} a_{i,j}}\big(\sqrt{\mu}g\big)\big |=\Big|\int_{\mathbb R^3} \big(\partial_{v_j} a_{i,j}\big) (v-v_*) \sqrt\mu (v_*) g (v_*)dv_*\Big|\leq C\comi v^{1+\gamma} \norm{g}_{L_v^2}
 \end{eqnarray*}
due to \eqref{coe}.   Similarly,   in view of Lemma \ref{Non-2},
\begin{align*}
	 \big|\inner{   L_5(g,h),\  \omega }_{L_v^2}\big| &=  \frac12\Big|\sum_{1\leq i,j\leq3}\inner{\partial_{v_i}\big(\bar a_{i,j}\big(v_j\sqrt{\mu}  g\big)h\big), \ \omega}_{L_v^2}+\sum_{1\leq i,j\leq3}\inner{ v_i \big( \bar a_{i,j}\big(\sqrt{\mu}\partial_{v_j}g\big)h\big), \ \omega}_{L_v^2}\Big|\\
	&  \leq 	C \norm{g}_{L_v^2}\norm{\psi(v,D_v)h}_{L_v^2}\norm{\psi(v,D_v)\omega}_{L_v^2}.
\end{align*}
 Combining the above estimates on $ \inner{   L_i(g,h),\  \omega }_{L_v^2}, 1\leq i\leq 6$, we   obtain \eqref{uptriestm} and thus  the trilinear estimate \eqref{trili} in Proposition \ref{prop: tricoe}.

 The rest part is devoted to proving the coercivity estimate \eqref{coer} in Proposition \ref{prop: tricoe}.  Recall $\mathcal Lh=-\Gamma(\sqrt \mu,h)-\Gamma(h, \sqrt \mu)$ in view of \eqref{gm}.  Using the trilinear estimate \eqref{trili} yields, for any $\delta>0,$
\begin{equation}\label{tri-non-2}
\begin{split}
\Big|\inner{ \Gamma(h, \sqrt \mu),\ h }_{L_v^2}\Big|
 \leq C\norm{h}_{L_v^2}\norm{\psi(v,D_v)h}_{L_v^2}
 \leq \delta\norm{\psi(v,D_v)h}_{L_v^2}^2+C_{\delta}\norm{h}^2_{L_v^2}.
\end{split}
\end{equation}
By Proposition \ref{proprep} we can write  \begin{equation*}\label{musp}
 - \Gamma(\sqrt\mu, h)=- \sum_{1\leq j\leq 6}L_j(\sqrt\mu, h)
\end{equation*}
with $L_j(\sqrt\mu, h)$ given in Proposition \ref{proprep}. Next we will proceed to derive the lower or upper bounds of $-\inner{L_i(\sqrt \mu, h),\ h}_{L_v^2}$  with $1\leq i\leq 6.$

Observe $v\wedge\partial_v$ is anti-selfadjoint on $L_v^2$. Then we use integration  by parts to obtain
\begin{multline*}
	 \frac{1}{2}\inner{\big (\partial_v\wedge\boldsymbol{B}_g\big )
\cdot \inner{v\wedge \partial_v}h- (v\wedge\partial_v)\cdot(\boldsymbol{B}_{g}\wedge\partial_v)h,\ h}_{L^2}\\
=-\frac{1}{2}\inner{
  \inner{v\wedge \partial_v}h,\ (\boldsymbol{B}_g\wedge\partial_v \big )  h}_{L^2}+\frac{1}{2}\inner{
 (\boldsymbol{B}_{g}\wedge\partial_v)h,\ (v\wedge\partial_v)  h}_{L^2}=0,
 \end{multline*}
which, with
  the representation of $L_1(\sqrt\mu, h)$ in  \eqref{defl}, yields
\begin{align*}
	-\inner{L_1(\sqrt \mu, h),\ h}_{L_v^2}&= \frac{1}{2} \int_{\mathbb R_v^3} \Big( a_{\sqrt\mu}(v)\inner{v\wedge\partial_v}h\Big)\cdot\inner{v\wedge \partial_v} h \,dv
+\sum_{1 \leq i,j\leq 3}\int_{\mathbb R_v^3} M_{i,j,\sqrt\mu}(v)\big( \partial_{v_j} h\big)\partial_{v_i}h \,dv.
\end{align*}
Here  $a_{\sqrt\mu} $ and  $M_{i,j,\sqrt\mu}$ are defined in \eqref{ag} and \eqref{M} which satisfy  that
 \begin{equation*}
a_{\sqrt\mu}(v)=\int_{\mathbb R^3}\abs{v-v_*}^{\gamma}\mu(v_*)dv_*\geq  \comi v^\gamma /C
\end{equation*}
and that
\begin{equation}\label{defmatr}
  \sum_{1\leq i,j\leq 3}M_{i,j,\sqrt\mu}(v)\zeta_i\zeta_j = \sum_{1\leq i,j\leq 3}\int_{\mathbb R^3}\abs {v-v_*}^\gamma  \big (\delta_{i,j}\abs{v_*}^2-(v_*)_i(v_*)_j\big ) \mu(v_*)dv_*
\geq  \comi v^{\gamma}\abs{\zeta}^2/C.	
\end{equation}
for any  $\zeta=(\zeta_1,\zeta_2,\zeta_3)\in\mathbb R^3$,
since the matrix $\big(\delta_{i,j}\abs{v_*}^2-(v_*)_i(v_*)_j\big)_{3\times 3}$ is positive-defined.   As a result, combining the above estimate yields
\begin{equation}
	\label{lol1}
	-\inner{L_1(\sqrt \mu, h),\ h}_{L_v^2}\geq  C^{-1}\inner{\norm{\comi v^{\gamma/2}(v\wedge\partial_v)h}_{L_v^2}^2+\norm{\comi v^{\gamma/2} \partial_v h}_{L_v^2}^2}.
\end{equation}
Direct verification shows
\begin{multline*}
	 \frac{1}{4}\inner{\inner{v\wedge \partial_v}\cdot\big ( \boldsymbol{B}_g\wedge v\big )
  h+  \big ( \boldsymbol{B}_g\wedge v\big )\cdot \inner{v\wedge \partial_v} h,\ h}_{L^2}\\
=-\frac{1}{4}\inner{ (\boldsymbol{B}_g\wedge v \big )
h,\     \inner{v\wedge \partial_v}h}_{L^2}+\frac{1}{4}\inner{
 \inner{v\wedge \partial_v}h,\  (\boldsymbol{B}_g\wedge v \big )  h}_{L^2}=0.
 \end{multline*}
 This implies, recalling $L_2(\sqrt\mu,h)$ is given in Proposition \ref{proprep} and the matrx $M_{i,j,\sqrt\mu}$ in \eqref{M} is  symmetric,
 \begin{multline}\label{upl2}
 	-\inner{L_2(\sqrt \mu, h),\ h}_{L_v^2} = \frac{1}{2}\sum_{1\leq i,j\leq3}\Big(
\Big[\partial_{v_i }(M_{i,j,\sqrt\mu}v_j h)+v_iM_{i,j,\sqrt\mu}\partial_{v_j}h\Big] ,  \  h \Big)_{L_v^2}\\
 = -\frac{1}{2}\sum_{1\leq i,j\leq3}\Big(
 M_{i,j,\sqrt\mu}v_j h ,  \   \partial_{v_i } h \Big)_{L_v^2}+\frac{1}{2}\sum_{1\leq i,j\leq3}\Big(
\partial_{v_j}h,  \  M_{i,j,\sqrt\mu}v_i  h \Big)_{L_v^2}=0.
 \end{multline}
 By the definitions of  $\lambda_{i,j,\sqrt\mu}$  and $M_{i,j,\sqrt\mu}$ in \eqref{M},
\begin{align*}
\sum_{\stackrel{1\leq i,j\leq3}{i\neq j}} v_i \lambda_{i,j,\sqrt\mu}(v)&=\sum_{\stackrel{1\leq i,j\leq3}{i\neq j}}  \int_{\mathbb R^3}\abs{v-v_*}^{\gamma} \mu(v_*)  \inner{v_i^2(v_*)_j(v_*)_j -v_iv_j(v_*)_i (v_*)_j }dv_*\\
&=\sum_{ 1\leq i,j \leq3}  \int_{\mathbb R^3} \abs{v-v_*}^{\gamma} \mu(v_*)    \inner{\delta_{i,j} |v_*|^2  - (v_*)_i (v_*)_j }  v_i  v_j  dv_*=\sum_{ 1\leq i,j \leq3}M_{i,j,\sqrt\mu} v_iv_j,
\end{align*}
the second equality using a similar fact as in \eqref{aiij}.  As a result,
\begin{equation}\label{posterm}
	\sum_{\stackrel{1\leq i,j\leq3}{i\neq j}}\inner{v_i \lambda_{i,j,\sqrt\mu} h, \ h}_{L_v^2}=\sum_{ 1\leq i,j \leq3}\inner{v_i M_{i,j,\sqrt\mu} v_j h, \ h}_{L_v^2},
\end{equation}
which with Proposition \ref{proprep} yields
\begin{multline}\label{l3l6est}
		-\inner{L_3(\sqrt \mu, h),\ h}_{L_v^2}- \inner{L_6(\sqrt \mu, h),\ h}_{L_v^2}\\
		=-\frac14\sum_{ 1\leq i,j \leq3}\inner{v_i M_{i,j,\sqrt\mu} v_j h, \ h}_{L_v^2}+\frac14	\sum_{\stackrel{1\leq i,j\leq3}{i\neq j}}\inner{v_i \lambda_{i,j,\sqrt\mu} h, \ h}_{L_v^2}=0.
\end{multline}
In view of  \eqref{ag} and \eqref{M},   we can verify that  \begin{eqnarray*}
\forall\ v\in\mathbb R^3,\quad \boldsymbol{A}_{\sqrt\mu }(v)=-\frac{1}{2} \boldsymbol{B}_{\sqrt\mu }(v) \  \textrm{ and }\ 	\rho_{i,j,\sqrt\mu }(v)=-\frac12  \lambda_{i,j,\sqrt\mu}(v),
\end{eqnarray*}
and thus, in view of the definitions of $L_4(\sqrt\mu, h)$ and $L_5(\sqrt\mu, h)$   in Proposition \ref{proprep},
\begin{multline}\label{l45}
-\inner{L_4(\sqrt \mu, h),\ h}_{L_v^2}- \inner{L_5(\sqrt \mu, h),\ h}_{L_v^2}\\
		 = \frac12 \inner{\inner{v\wedge \partial_v}\cdot\inner{\boldsymbol{B}_{\sqrt{\mu}}\wedge v} h, \ h}_{L_v^2}-	\sum_{\stackrel{1\leq i,j\leq3}{i\neq j}}\inner{\partial_{v_i} \big(\lambda_{i,j,\sqrt{\mu}} h\big), \ h}_{L_v^2}+\frac14	\sum_{\stackrel{1\leq i,j\leq3}{i\neq j}}\inner{v_i \lambda_{i,j,\sqrt{\mu}} h, \ h}_{L_v^2}.
\end{multline}
As for the last term on the right side of \eqref{l45}, we use \eqref{posterm} and  \eqref{defmatr} to conclude
\begin{equation}\label{lowlam}
	\frac14	\sum_{\stackrel{1\leq i,j\leq3}{i\neq j}}\inner{v_i \lambda_{i,j,\sqrt{\mu}} h, \ h}_{L_v^2}\geq C^{-1} \int_{\mathbb R^3}\comi v^\gamma \abs{v}^2h(v)^2dv.
\end{equation}
Integrating by parts and using the fact that
\begin{eqnarray*}
 \partial_{v_i}\lambda_{i,j,\sqrt{\mu}} = \partial_{v_i}	\Big[v_i \big(\abs v^\gamma* (\mu v_j^2)\big)-v_j \big(\abs v^\gamma* (\mu v_iv_j )\big) \Big]   \geq  -C\comi v^{1+\gamma}
\end{eqnarray*}
in view of \eqref{M}, we have,
 for any $\delta>0,$
\begin{equation}\label{lowbounla}
	 -\sum_{\stackrel{1\leq i,j\leq3}{i\neq j}}\inner{\partial_{v_i} \big(\lambda_{i,j,\sqrt{\mu}} h\big), \ h}_{L_v^2}=\frac12 \sum_{\stackrel{1\leq i,j\leq3}{i\neq j}}\inner{ \big(\partial_{v_i}\lambda_{i,j,\sqrt{\mu}}  \big)h, \ h}_{L_v^2}
	 \geq  -\delta  \norm{\comi v^{1+{\gamma\over2}}h}_{L_v^2}^2- C_\delta \norm{ h}_{L_v^2}^2.
	\end{equation}
Finally for the first term on the right hand side of \eqref{l45},  observe  $\boldsymbol{B}_{\sqrt{\mu}}=-\partial_v a$  with $a=a_{\sqrt\mu}$, and thus, writing $a$ instead of $a_{\sqrt\mu}$ for simplicity of notations,
\begin{align*}
& \frac12  \inner{v\wedge \partial_v}\cdot\big(\boldsymbol{B}_{\sqrt{\mu}}\wedge v\big) h =  -\frac12 \sum_{\stackrel{1\leq i,j\leq 3}{i\neq j}} \Big( v_i\partial_{v_j} -v_j\partial_{v_i}\Big)\inner{\big(\partial_{v_i}a\big)v_j h-\big(\partial_{v_j}a\big)v_i h  }  \\&
=-\frac12 \sum_{\stackrel{1\leq i,j\leq 3}{i\neq j}} \inner{v_i\big(\partial_{v_i}\partial_{v_j}a\big)v_j h+v_i\big(\partial_{v_i} a\big) h+v_i\big(\partial_{v_i} a\big)v_j\partial_{v_j} h-v_i\big(\partial_{v_j}^2a\big)v_i h-v_i\big(\partial_{v_j}a\big)v_i \partial_{v_j}h  } \\
&\quad +\frac12\sum_{\stackrel{1\leq i,j\leq 3}{i\neq j}} \inner{v_j\big(\partial_{v_i}^2a\big)v_j h+v_j\big(\partial_{v_i}a\big)v_j \partial_{v_i}h-v_j\big(\partial_{v_i}\partial_{v_j}a\big)v_i h-v_j\big(\partial_{v_j}a\big)  h-v_j\big(\partial_{v_j}a\big)v_i\partial_{v_i} h  }\\
&=\sum_{\stackrel{1\leq i,j\leq 3}{i\neq j}}\Big( \big(\partial_{v_i}^2a\big)v_j ^2h+ \big(\partial_{v_i}a\big)v_j^2 \partial_{v_i}h - \big(\partial_{v_i}\partial_{v_j}a\big)v_i v_jh- \big(\partial_{v_j}a\big) v_j h -\big(\partial_{v_j}a\big)v_iv_j\partial_{v_i} h  \Big).
\end{align*}
This, with the fact that
\begin{eqnarray*}
	\inner{\big(\partial_{v_i}a\big)v_j^2 \partial_{v_i}h,\ h}_{L_v^2}=-\frac{1}{2}\inner{\big(\partial_{v_i}^2a\big)v_j^2  h,\ h}_{L_v^2} \end{eqnarray*}
and
\begin{eqnarray*}
	  \inner{-\big(\partial_{v_j}a\big)v_iv_j\partial_{v_i} h,\ h}_{L_v^2}= \frac12\inner{ \big(\partial_{v_i} \partial_{v_j}a\big)v_iv_jh,\ h}_{L_v^2}+\frac12\inner{ \big( \partial_{v_j}a\big) v_jh,\ h}_{L_v^2}
\end{eqnarray*}
for any $i\neq j$, yields
\begin{multline}\label{laetimate}
\frac12 \inner{ \inner{v\wedge \partial_v}\cdot\big(\boldsymbol{B}_{\sqrt{\mu}}\wedge v\big) h,\ h}_{L_v^2} =\frac12 \sum_{\stackrel{1\leq i,j\leq 3}{i\neq j}}\inner{\big(\partial_{v_i}^2a\big)v_j ^2h-\big(\partial_{v_i}\partial_{v_j}a\big)v_i v_jh-\big(\partial_{v_j}a\big) v_j h ,\ h}_{L_v^2}\\
 \geq  \frac12 \sum_{\stackrel{1\leq i,j\leq 3}{i\neq j}}\inner{ -\big(\partial_{v_j}a\big) v_j h ,\ h}_{L_v^2}\geq -C\norm{\comi v^{(1+\gamma)/2}h}_{L_v^2}^2\geq -\delta  \norm{\comi v^{1+{\gamma\over2}}h}_{L_v^2}^2- C_\delta \norm{ h}_{L_v^2}^2,
\end{multline}
where the first inequality in the last line  holds true because it follows from \eqref{defmatr} that, recalling $a=a_{\sqrt\mu}$ is defined in \eqref{ag},
\begin{align*}
	\sum_{\stackrel{1\leq i,j\leq 3}{i\neq j}}\Big(\big(\partial_{v_i}^2a\big)v_j ^2 -\big(\partial_{v_i}\partial_{v_j}a\big)v_i v_j \Big)&=\sum_{\stackrel{1\leq i,j\leq 3}{i\neq j}}\int_{\mathbb R^3}\abs{v-v_*}^\gamma\mu(v_*)\Big[ (v_*)_i(v_*)_i v_j^2  -  (v_*)_i(v_*)_j v_iv_j \Big]   dv_*  \\
	&=\sum_{ 1\leq i,j\leq 3 }\int_{\mathbb R^3}\abs{v-v_*}^\gamma\mu(v_*)\Big[\delta_{i,j}|v_*|^2  -  (v_*)_i(v_*)_j\Big]  v_iv_j   dv_* \geq 0.
	\end{align*}
Now we substitute \eqref{lowlam}, \eqref{lowbounla} and \eqref{laetimate} into \eqref{l45} to conclude
\begin{eqnarray*}
 -\inner{L_4(\sqrt \mu, h),\ h}_{L_v^2}- \inner{L_5(\sqrt \mu, h),\ h}_{L_v^2}\geq  C^{-1} \int_{\mathbb R^3}\comi v^\gamma \abs{v}^2h(v)^2dv-\delta  \norm{\comi v^{1+{\gamma\over2}}h}_{L_v^2}^2- C_\delta \norm{ h}_{L_v^2}^2
\end{eqnarray*}
for any $\delta>0$, that is,
\begin{eqnarray*}
 -\inner{L_4(\sqrt \mu, h),\ h}_{L_v^2}- \inner{L_5(\sqrt \mu, h),\ h}_{L_v^2}\geq  C^{-1} \norm{\comi v^{1+{\gamma\over2}}h}_{L_v^2}^2- C  \norm{ h}_{L_v^2}^2.
\end{eqnarray*}
This, with \eqref{lol1}, \eqref{upl2} and \eqref{l3l6est} as well as \eqref{ma+0},  implies
\begin{eqnarray*}
	-\inner{\Gamma(\sqrt\mu, h),\ h}_{L_v^2}\geq  C^{-1} \norm{\psi(v,D_v)h}_{L_v^2}^2- C  \norm{ h}_{L_v^2}^2.
\end{eqnarray*}
As a result, the coercivity estimate \eqref{coer} follows by   combining the above estimate with \eqref{tri-non-2} and observing $\mathcal Lh=-\Gamma(\sqrt \mu,h)-\Gamma(h, \sqrt \mu)$.  The proof of Proposition \ref{prop: tricoe} is completed.

\section{Estimate on commutators}

 This part is devoted to treating the commutators between $M^k$ and $ \Gamma(f,f)$,  and completing the proof of Proposition \ref{prop:com}.

\subsection{Quantitative properties of the  time-average operator}

Let $M$ be the  time-average operator defined by \eqref{timave}, which is a Fourier multiplier with symbol
\begin{eqnarray*}
	(t-t_0)\eta_1^2+(t-t_0)^2m_1\eta_1+\frac{(t-t_0)^3}{3}m_1^2,
\end{eqnarray*}
recalling $m=(m_1,m_2,m_3)\in\mathbb Z^3$ and $\eta=(\eta_1,\eta_2,\eta_3)\in\mathbb R^3$ are the Fourier dual variables of $x\in\mathbb T^3$ and $v\in\mathbb R^3$, respectively.

Direct verification yields, for any $t>t_0,$
\begin{equation}\label{elliptic}
 c_0\Big( (t-t_0)\eta_1^2+ (t-t_0)^3m_1^2\Big)\leq  (t-t_0)\eta_1^2+ (t-t_0)^2m_1\eta_1+\frac{ (t-t_0)^3}{3}m_1^2\leq  \Big( (t-t_0)\eta_1^2+ (t-t_0)^3m_1^2\Big)/c_0
\end{equation}
for some constant $0<c_0<1$. This enables to define the fractional power $M^\sigma$ by setting
\begin{equation}\label{FoMu}
\forall\ \sigma\geq 0	,\quad 	\widehat {M^\sigma  g}(m,\eta)=\rho_\sigma\hat g(m,\eta)
\end{equation}
with Fourier symbol
\begin{equation}\label{rhsigma}
\rho_\sigma:=	\Big( (t-t_0)\eta_1^2+ (t-t_0)^2m_1\eta_1+\frac{ (t-t_0)^3}{3}m_1^2\Big)^\sigma, \quad \sigma\geq 0, \ t\in [t_0,1].
\end{equation}

 \begin{lemma}[symbolic calculus]\label{lem:syM}
For any given integer $k\geq 0$ we have
 \begin{equation}\label{sycal}
\forall\ t\in[t_0,1],\  \forall \   j\leq 2k,\quad |\partial_{\eta_1}^j \rho_k|\leq 8^j \frac{(2k)!}{(2k-j)!} \rho_{k-\frac{j}{2}},
\end{equation}
where $\rho_k$ is defined by \eqref{rhsigma}.
\end{lemma}

\begin{proof}   We use induction on $k$ to prove \eqref{sycal}, which  holds true for $k=0$ or $k=1$, since direct computation shows
	\begin{eqnarray*}
		\partial_{\eta_1}\rho_1=2(t-t_0)\eta_1+(t-t_0)^2m_1, \quad \partial_{\eta_1}^2\rho_1=2(t-t_0),
	\end{eqnarray*}
	and moreover   $\abs{2(t-t_0)\eta_1+(t-t_0)^2m_1}\leq 4 \rho_{1/2}$ and $|2(t-t_0)|\leq 2\rho_0$ for $t\in[t_0,1]$.
	 Now  suppose   for any $i \leq k-1$ with given integer $k\geq 2,$  we have
	\begin{equation}\label{ina}
\forall\ t\in[t_0,1],\ \forall~j\leq 2i, \quad 	\abs{\partial_{\eta_1}^j \rho_i}\leq 8^j \frac{(2i)!}{(2i-j)!} \rho_{i-\frac{j}{2}},
\end{equation}
we will show the above estimate also holds for 	  $i=k$, that is,
	\begin{equation}\label{ina+}
\forall\ t\in[t_0,1],\  \forall~j\leq 2k, \quad 	\abs{\partial_{\eta_1}^j \rho_k}\leq 8^j \frac{(2k)!}{(2k-j)!} \rho_{k-\frac{j}{2}}.
\end{equation}
Note \eqref{ina+} holds true for $j=2k$, since  \begin{equation}\label{j2k}
\forall\ t\in[t_0,1],\quad  	\partial_{\eta_1}^{2k} \rho_k=(2k)! (t-t_0)^k\leq  (2k)!  \rho_{0}.
	\end{equation}
Now we consider the case  when $2\leq j\leq 2k-1$ and  write
	\begin{multline*}
		\partial_{\eta_1}^j \rho_k=	\partial_{\eta_1}^{j-1} \big[k\rho_{k-1}\big (2(t-t_0)\eta_1+(t-t_0)^2m_1\big )\big] \\
		= k(\partial_{\eta_1}^{j-1} \rho_{k-1} )\big(2(t-t_0)\eta_1+(t-t_0)^2m_1\big)+k(j-1)(\partial_{\eta_1}^{j-2} \rho_{k-1} )2(t-t_0).
	\end{multline*}
	Moreover observe $j-2\leq j-1\leq 2(k-1)$, and thus
	we use the induction assumption \eqref{ina}  as well as the fact that  $\abs{2(t-t_0)\eta_1+(t-t_0)^2m_1}\leq 4 \rho_{1/2}$ to compute
	\begin{eqnarray*}
		\abs{ k(\partial_{\eta_1}^{j-1} \rho_{k-1} ) \big(2(t-t_0)\eta_1+(t-t_0)^2m_1\big)}\leq
		k 8^{j-1}   \frac{[2(k-1)]!}{(2k-j-1)!} \rho_{k-1-\frac{j-1}{2}}\times 4\rho_{1/2} =\frac{8^{j}}{2} \frac{(2k)!}{(2k-j)!} \rho_{k- \frac{j}{2}}
	\end{eqnarray*}
	and
	\begin{eqnarray*}
	\abs{k(j-1)(\partial_{\eta_1}^{j-2} \rho_{k-1} )2(t-t_0)} \leq  k(j-1)  8^{j-2} \frac{[2(k-1)]!}{(2k-j)!}\rho_{k-1-\frac{j-2}{2}} 2(t-t_0)  \leq
		\frac{8^{j}}{2}  \frac{(2k)!}{(2k-j)!} \rho_{k- \frac{j}{2}}.	
	\end{eqnarray*}
	Combining these inequalities we obtain
	\begin{eqnarray*}
\forall\ t\in[t_0,1],\ 	\forall\ 2\leq  j\leq 2k-1,\quad	\abs{\partial_{\eta_1}^j \rho_k} \leq  8^{j} \frac{(2k)!}{(2k-j)!} \rho_{k- \frac{j}{2}}.
	\end{eqnarray*}
	Finally, direct verification shows
	\begin{eqnarray*}
\forall\ t\in[t_0,1],\ 	\forall\ 0\leq  j \leq 1,\quad	\abs{\partial_{\eta_1}^j \rho_k} \leq  8^{j} \frac{(2k)!}{(2k-j)!} \rho_{k- \frac{j}{2}}.
	\end{eqnarray*}
Combining the above two estimates  with \eqref{j2k} yields  \eqref{ina+}.  The proof of Lemma \ref{lem:syM} is   completed.
\end{proof}

   By direct computation  we have, for any $t\in[t_0,1],$
\begin{multline*}
		(t-t_0)\eta_1^2+(t-t_0)^2m_1\eta_1+\frac{(t-t_0)^3}{3}m_1^2\\
		=\frac{1}{4}\Big((t-t_0)^{1/2}\eta_1+(t-t_0)^{3/2}m_1\Big)^2+\frac{1}{12}\Big (3(t-t_0)^{1/2}\eta_1+(t-t_0)^{3/2}m_1
		\Big )^2.
\end{multline*}
Then
   $M$ can be represented  as the  square sum of a vector filed $\Lambda=(\Lambda_1,\Lambda_2)$:
 \begin{equation*}
M= \Lambda \, \cdot\, \Lambda=   \Lambda^2_1+\Lambda_2^2,
\end{equation*}
where, letting $\sqrt{-1}$ be the square root of $-1$,
\begin{equation}\label{sqroot}
\Lambda_1= \frac{1}{2\sqrt{-1}}\Big( (t-t_0)^{1/2}\partial_{v_1}+   (t-t_0)^{3/2}\partial_{x_1}\Big),\quad \Lambda_2= \frac{\sqrt 3}{6\sqrt{-1}}\Big(3 (t-t_0)^{1/2}\partial_{v_1}+   (t-t_0)^{3/2}\partial_{x_1}\Big).
\end{equation}
Note $\Lambda_i, 1\leq i\leq 2,$ are self-adjoint operators in $H^{(2,0)}.$
  By virtue of the vector field   $\Lambda$ we have
  the following Leibniz  type formula.

\begin{lemma}\label{lem: leibniz}
For any $k\in\mathbb{Z}_+$, it holds that
\begin{align}\label{Leb}
M^k( gh)=(\Lambda \, \cdot\, \Lambda)^k(g,h)=\sum^{2k}_{j=0}A_{j, 2k-j}(g,h),
\end{align}
with
\begin{align}\label{aj2k}
A_{j, 2k-j}(g,h)
&=\sum_{\ell+2p=j,\ \ell+2q=2k-j}c^{k, j}_{\ell, p, q} (\Lambda^\ell  M^pg)\,\cdot\, (\Lambda^{\ell } M^qh),
\end{align}
where   the summation is taken over   all non-negative integers $\ell, p $ and $q $ satisfying $\ell+2p=j$ and $\ell+2q=2k-j$, and
\begin{align}\label{prod}
\Lambda^\ell g\cdot\, \Lambda^\ell h:=\sum^{2}_{j_1=1}\cdot\cdot\cdot\sum^{2}_{j_\ell=1}(\Lambda_{j_1}\cdots\Lambda_{j_\ell}g)\, (\Lambda_{j_1}\cdots\Lambda_{j_\ell}h).
\end{align}
The sequence $ c^{k, j}_{\ell, p, q} $
of non-negative integers in \eqref{aj2k} are determined by
\begin{equation}\label{rela}
c^{k+1, j}_{\ell, p, q}=	c^{k, j-2}_{\ell, p-1, q}+c^{k, j}_{\ell, p, q-1}+2c^{k, j-1}_{\ell-1, p, q}
\end{equation}
where we have used the convention that
\begin{equation}\label{conven}
	c^{k,j}_{\ell,p,q} =0\  \textrm { if }  j>2k  \textrm{ or  any one entry  in  the index }   (j,\ell, p, q)  \textrm{ is } negative.
\end{equation}
Moreover
\begin{equation}\label{sum}
	\sum_{\ell+2p=j,\ \ell+2q=2k-j}c^{k, j}_{\ell, p, q}={{2k}\choose j}.
\end{equation}
\end{lemma}

\begin{proof}
We use induction on $k$ to prove \eqref{Leb}. The validity of formula \eqref{Leb} for $k=0$ is obvious. For given integer $k\geq 0$, suppose that
\begin{equation}\label{ias}
\forall\ N\leq k ,\quad	M^N( gh)=\sum^{2N}_{j=0}A_{j, 2N-j}(g,h),
\end{equation}
with $A_{j, 2N-j}(g,h)$ defined in \eqref{aj2k}.
We will prove the above assertion still holds true for $N=k+1.$

It follows from the inductive assumption \eqref{ias} that
\begin{align}\label{indstep}
M^{k+1}(gh)&=MM^{k}(gh)=\sum^{2k}_{j=0}M A_{j, 2k-j}(g,h).
\end{align}
Moreover  in view of \eqref{aj2k} and the fact
$$
 M(gh)=\Lambda\cdot\Lambda (gh)=(M g)h+g (M h)+2(\Lambda  g) \cdot  (\Lambda  h),
$$
we compute
\begin{align*}
& \sum^{2k}_{j=0}MA_{j, 2k-j}(g,h) =  \sum^{2k}_{j=0}
\sum_{\stackrel{ \ell+2p=j,}{ \ell+2q=2k-j}}c^{k, j}_{\ell, p, q}(\Lambda^\ell  M^{p+1}g)\,\cdot\, (\Lambda^{\ell } M^qh)\\
&\qquad+ \sum^{2k}_{j=0}
\sum_{\stackrel{\ell+2p=j}{ \ell+2q=2k-j}}c^{k, j}_{\ell, p, q}(\Lambda^\ell  M^pg)\,\cdot\, (\Lambda^{\ell } M^{q+1}h)
+  \sum^{2k}_{j=0}
\sum_{\stackrel{ \ell+2p=j}{ \ell+2q=2k-j}}2c^{k, j}_{\ell, p, q}(\Lambda^{\ell+1}  M^pg)\,\cdot\, (\Lambda^{\ell+1} M^qh)\\
& =  \sum^{2(k+1)}_{j=2}
\sum_{\stackrel{ \ell+p=j,}{ \ell+2q=2(k+1)-j}}c^{k, j-2}_{\ell, p-1, q}(\Lambda^\ell  M^{p}g)\,\cdot\, (\Lambda^{\ell } M^qh)+ \sum^{2k}_{j=0}
\sum_{ \stackrel{\ell+2p=j}{ \ell+2q=2(k+1)-j}}c^{k, j}_{\ell, p, q-1}(\Lambda^\ell  M^pg)\,\cdot\, (\Lambda^{\ell } M^{q}h)
\\&
\qquad+  \sum^{2k+1}_{j=1}
\sum_{\stackrel{ \ell+2p=j}{  \ell+2q=2(k+1)-j}}2c^{k, j-1}_{\ell-1, p, q}(\Lambda^{\ell}  M^pg)\,\cdot\, (\Lambda^{\ell} M^qh)\\
 &= \sum_{j=0}^{2(k+1)}\sum_{ \stackrel{\ell+2p=j}{ \ell+2q=2(k+1)-j}}c^{k+1, j}_{\ell, p}(\Lambda^\ell  M^pg)\,\cdot\, (\Lambda^{\ell } M^{q}h)=\sum_{j=0}^{2(k+1)}A_{j, 2(k+1)-j}(g,h),
\end{align*}
where in the last line we have used \eqref{rela} with the convention   \eqref{conven}.  As a result, substituting the above equations into \eqref{indstep} yields the validity of \eqref{ias} for $N=k+1$.  We have proven  \eqref{Leb}.

It remains to prove the last assertion \eqref{sum}. We use  induction on $k$.  If $k=0$ then $j=0$ and
thus $A_{0,0}(g,h)=gh$ in view of \eqref{Leb}. Meanwhile by \eqref{aj2k},
\begin{eqnarray*}
	A_{0,0}(g,h)=c^{0, 0}_{0, 0, 0} gh,
\end{eqnarray*}
which gives $c^{0, 0}_{0, 0, 0}=1$. As a result, for $k=0$ we have
\begin{eqnarray*}
	\sum_{\ell+2p=j,\ \ell+2q=2k-j}c^{k, j}_{\ell, p, q}=c^{0, 0}_{0, 0, 0}=1={{2k}\choose j}.
\end{eqnarray*}
Then
\eqref{sum} holds true for $k=0$. Supposing \eqref{sum} holds for any integer $k$ with $k\geq 1$, we will prove its validity for $k+1$.  In fact we use \eqref{rela} with the convention \eqref{conven}, to compute
\begin{align*}
 \sum_{\stackrel{\ell+2p=j}{\ell+2q=2(k+1)-j}}c^{k+1, j}_{\ell, p, q}
& =	\sum_{\stackrel{\ell+2p=j}{\ell+2q=2(k+1)-j}}c^{k, j-2}_{\ell, p-1, q}+\sum_{\stackrel{\ell+2p=j}{\ell+2q=2(k+1)-j}}c^{k, j}_{\ell, p, q-1}+\sum_{\stackrel{\ell+2p=j}{\ell+2q=2(k+1)-j}}2c^{k, j-1}_{\ell-1, p, q}
\\
&=\sum_{\stackrel{\ell+2p=j-2}{ \ell+2q=2k-(j-2)}}c^{k, j-2}_{\ell, p, q}
+
\sum_{\stackrel{\ell+2p=j}{\ell+2q=2k-j}}c^{k, j}_{\ell, p, q}+
2\sum_{\stackrel{\ell+2p=j-1}{ \ell+2q=2k-(j-1)}}c^{k, j-1}_{\ell, p, q}\\
&=
{{2k}\choose{j-2}}+{{2k}\choose j}+2{{2k}\choose{j-1}}
={{2(k+1)}\choose j}.
\end{align*}
Thus \eqref{sum} holds for all $k\geq 0$. The proof of
 Lemma \ref{lem: leibniz} is completed.
\end{proof}

\begin{remark}
	\label{rem:spec} With  $c^{k, j}_{\ell, p, q}$ given in the above Lemma \ref{lem: leibniz},
	it follows from \eqref{sum} that
  $c^{k, 2k}_{0, k, 0}=c^{k, 2k}_{0, 0, k}=1$, and   $c^{k, 2k-1}_{1, k-1, 0}=c^{k, 1}_{1, 0, k-1}=2k.$
\end{remark}

\begin{remark}\label{rem:nota}
 In the following discussion,  by writing $\norm{\Lambda^\ell g} $  for some   generic norm  $\norm{\cdot}$ we mean
\begin{eqnarray*}
		\norm{ \Lambda^\ell g} =  \bigg(\sum_{1\leq j_1\leq 2}
		\cdots\sum_{1\leq j_\ell\leq 2}
\norm{\Lambda_{j_1}\cdots \Lambda_{j_\ell}g}^2\bigg)^{1/2}.
	\end{eqnarray*}
     Note $\norm{ \Lambda^2 g}\neq \norm{ M g}$. Moreover by Cauchy inequality it follows that
  	\begin{eqnarray*}
  		\norm{ \Lambda^\ell g\cdot \Lambda^\ell h}_{(2,0)}\leq \sum^{2}_{j_1=1}\cdot\cdot\cdot\sum^{2}_{j_\ell=1}\norm{ \Lambda_{j_1}\cdots\Lambda_{j_\ell}g }_{(2,0)} \norm{ \Lambda_{j_1}\cdots\Lambda_{j_\ell}h }_{H_x^2L_v^\infty}\leq \norm{ \Lambda^\ell g }_{(2,0)}\norm{  \Lambda^\ell h}_{H_x^2L_v^\infty},
  	\end{eqnarray*}
  	recalling 	$ \Lambda^\ell g\cdot \Lambda^\ell h $ is defined by \eqref{prod}.
  	\end{remark}

\begin{lemma}\label{MM}With the notations in Remark \ref{rem:nota},
it holds that
\begin{align*}
\|\Lambda^\ell g\|_{(2,0)}^2\leq
\left\{\begin{aligned}
&\|M^{\ell/2} g\|_{(2,0)}^2    \quad    \mbox{for even  number} \,\, \ell,    \\
&\|M^{\ell/2} g\|_{(2,0)}^2  \leq \norm{M^{(\ell+1)/2} g}_{(2,0)}\norm{M^{(\ell-1)/2}  g}_{(2,0)}  \quad \mbox{for odd number}  \,\, \ell.
\end{aligned}
\right.
\end{align*}
\end{lemma}

\begin{proof}
In view of  Remark \ref{rem:nota},
 we have
\begin{align*}
\| \Lambda^\ell  g \|_{(2,0)}^2 &=\sum^2 _{j_1=1}\,\,\cdots\,\,\sum^2_{j_\ell=1} \|\Lambda_{j_1}\cdots\Lambda_{j_\ell} g \|_{(2,0)}^2
=\sum^2 _{j_1=1}\,\,\cdots\,\,\sum^2_{j_\ell=1} \inner{\Lambda_{j_1}\cdots\Lambda_{j_\ell} g,\, \Lambda_{j_1}\cdots\Lambda_{j_\ell} g}_{(2,0)}
\\&
=\sum^2 _{j_1=1}\,\,\cdots\,\,\sum^2_{j_\ell=1} \inner{\Lambda^2_{j_1}\cdots\Lambda^2_{j_\ell} g,\,g }=\inner{M^\ell g,\,g }_{(2,0)},
\end{align*}
which implies the assertion.
The proof of lemma \ref{MM} is completed.
\end{proof}

\subsection{Proof of Proposition \ref{prop:com}: estimate on commutators} This part is devoted to estimating  the commutator between $M^k$ and the collision operator $\Gamma(f,f)$.  By Proposition \ref{proprep},
 $$\Gamma(f,f)=\sum_{1\leq j\leq 6} L_j(f,f)$$ with $L_j(f, f)$ the bilinear operators given in Proposition \ref{proprep}.
We proceed through  the following lemmas to deal with the commutator between $M^k$ and $L_j(f, f), 1\leq j\leq 6.$

In the following discussion we  will use $C\geq 1$  to denote a generic constant independent of $\epsilon, C_*$   and the derivative orders denoted by $k$, recalling $\epsilon, C_*$ are the constants given in the inductive assumption \eqref{ass:ind}.

\begin{lemma}
	[technical lemma]\label{lem:tec0}
	Suppose the  inductive assumption \eqref{ass:ind} in Proposition \ref{prop:com} holds.  Let $\phi=\phi(v)$ be a given function of $v$ variable only satisfying that
\begin{equation}\label{upgiven}
\exists\ L>0,\  \forall \ \beta\in\mathbb Z_+^3,\quad |\partial_v^\beta \phi(v)|\leq L^{\abs\beta+1}|\beta|!.
\end{equation}
Then
 \begin{eqnarray*}
    \sup_{t_0\leq t\leq 1}\norm{  M^i (\phi f) }_{(2,0)} \leq
 \left\{
 \begin{aligned}
 &C   \frac{\epsilon }{(2i+1)^3}  C_*^{ 2i} (2i)!,\quad \textrm{ if } \ 0\leq i\leq k-1,\\[2pt]
 & C \norm{M^{k} f }_{(2,0)}  +C   \frac{\epsilon }{(2k+1)^3}  C_*^{ 2k} (2k)!,\quad \textrm{ if } \ i=k,
  \end{aligned}
 \right.
   \end{eqnarray*}
 provided $C_*$ is large enough such that $C_*^{1/2}\geq 196L,$  where the constant $C$ depends on $L.$
\end{lemma}

\begin{proof}
The assertion for $i=0$ is obvious in view of the inductive assumption \eqref{ass:ind}. Now consider the case when $1\leq i\leq k.$
Using \eqref{FoMu} yields
\begin{align*}
	 \widehat{ M^i(\phi f)}(m,\eta)-\widehat{\phi M^i f}(m,\eta)
& =  \int_{\mathbb R^3}\hat\phi( \eta-\tau)\Big[ \rho_i(m_1,\eta_1)\hat f(m,\tau )  -\rho_i(m_1,\tau_1)\hat f(m,\tau ) \Big] d\tau\\
	&= \int_{\mathbb R^3}\hat \phi( \eta-\tau)  \sum_{1\leq j\leq 2k}\frac{(\partial_{\eta_1}^j\rho_{i})(m_1,\tau_1)}{j!}(\eta_1-\tau_1)^j \hat f(m,\tau ) d\tau \\
	&=  \sum_{1\leq j\leq 2i} \frac{1}{j!}\int_{\mathbb R^3}\widehat{ \partial_{v_1}^j\phi}( \eta-\tau)  (\partial_{\eta_1}^j\rho_{i})(m_1,\tau_1)   \hat f(m,\tau ) d\tau.
\end{align*}
This implies
\begin{equation}\label{instep}
\begin{aligned}
	 \norm{  M^{i} (\phi f) }_{(2,0)}\leq  \norm{ \phi M^{i} f }_{(2,0)}+  \sum_{1\leq j\leq 2i}\frac{1}{j!} \norm{(\partial_{v_1}^j\phi) P_j f }_{(2,0)},
	 \end{aligned}
\end{equation}
where   the operator $P_j$ in the last term stands for a Fourier multiplier with the symbol $\partial_{\eta_1}^j\rho_{i}$, that is,
\begin{eqnarray*}
	\widehat{ P_j f}(m,\eta)=\partial_{\eta_1}^{j} \rho_{i}(m_1,\eta_1) \hat f(m,\eta).
\end{eqnarray*}
As a result, we use  \eqref{sycal}
 and \eqref{upgiven} as well as \eqref{FoMu},   to conclude
\begin{equation*}
	\sum_{1\leq j\leq 2i}\frac{1}{j!} \norm{(\partial_{v_1}^j\phi) P_j f }_{(2,0)}\leq   \sum_{1\leq j \leq 2i} L^{j+1}8^j\frac{(2i)!}{(2i-j)!} \|   M^{i-\frac{j}{2}}f  \|_{(2,0)}.\end{equation*}
Substituting the above estimate into \eqref{instep} yields
\begin{equation*}\label{medine}
	 \norm{  M^{i} (\phi f) }_{(2,0)} \leq  C \norm{   M^{i} f }_{(2,0)}+16L^2 \,i\, \|   M^{i-\frac{1}{2}}f  \|_{(2,0)}+\sum_{2\leq j \leq 2i} L^{j+1}8^j\frac{(2i)!}{(2i-j)!} \|   M^{i-\frac{j}{2}}f  \|_{(2,0)}
\end{equation*}
Next we deal with
  the last term on the right hand side.  When $1\leq i\leq k$, we use inductive assumption \eqref{ass:ind}  to conclude that  for any integer  $j$ with $2\leq j\leq 2i$, if  $j$ is even then
\begin{eqnarray*}
	\|   M^{i-\frac{j}{2}}f  \|_{(2,0)} \leq \frac{\epsilon}{(2i-j+1)^3} C_*^{2i-j}(2i-j)!,
\end{eqnarray*}
and meanwhile  if
  $j$ is odd then
\begin{align*}
&	 \|   M^{i-\frac{j}{2}}f  \|_{(2,0)} \leq   \|   M^{i-\frac{j+1}{2}}f  \|_{(2,0)}^{1/2} \|   M^{i-\frac{j-1}{2}}f  \|_{(2,0)}^{1/2}\\
	 &\leq  C\frac{\epsilon}{(2i-j+1)^3} \Big[C_*^{2i-j-1}(2i-j-1)!C_*^{2i-j+1}(2i-j+1)!\Big]^{1\over2} \leq C\frac{\epsilon}{(2i-j+1)^3} C_*^{2i-j}(2i-j)!
\end{align*}
due to Lemma \ref{MM}.
Thus we compute
 \begin{equation}\label{suet}
  \begin{aligned}
  		&\sum_{2\leq j \leq 2i} L^{j+1}8^j\frac{(2i)!}{(2i-j)!} \|   M^{i-\frac{j}{2}}f  \|_{(2,0)}
		 \leq C\epsilon (2i)!   \sum_{2\leq j \leq 2i} L^{j+1} 8^j  \frac{1}{(2i-j+1)^3}C_*^{ 2i-j}\\
		&\leq C\epsilon (2i)!   \sum_{2\leq j \leq  i} L^{j+1} 8^j  \frac{1}{(2i-j+1)^3}C_*^{ 2i-j} + C\epsilon (2i)!   \sum_{   i+1\leq j\leq 2i} L^{j+1} 8^j  \frac{1}{(2i-j+1)^3}C_*^{ 2i-j}\\
		&\leq C   \frac{\epsilon }{(2i+1)^3}  C_*^{ 2i} (2i)!   \bigg[\sum_{2\leq j\leq i} (8 L)^{j+1}  C_*^{  -j}+ \sum_{   i+1\leq j\leq 2i} (8L)^{j+1}     (2i+1)^3 C_*^{-j}\bigg]\\
		&\leq C   \frac{\epsilon }{(2i+1)^3}  C_*^{ 2i} (2i)!,
  \end{aligned}
  \end{equation}
  the last line using the estimates that
  \begin{eqnarray*}
  \sum_{2\leq j\leq i }(8 L)^{j+1}  C_*^{  -j} \leq C
  \end{eqnarray*}
  and
  \begin{align*}
   \sum_{   i+1\leq j\leq   2i} (8L)^{j+1}     (2i+1)^3 C_*^{ -j} &\leq  	\sum_{   i+1\leq j\leq   2i}(8L)^{2i+1}    (2i+1)^3 C_*^{ -i-1} \\
   &\leq  (8L)^{2i+1}    (2i+1)^4 C_*^{ -i-1}\leq   (192L)^{2i+1}     C_*^{ -i-1}\leq C,
  \end{align*}
  provided $\sqrt{C_*}\geq 192L.$ Thus combining \eqref{suet} and \eqref{instep} we conclude, for any $1\leq i\leq k,$
  \begin{equation*}
  \begin{aligned}
  \norm{  M^{i} (\phi f) }_{(2,0)} &\leq  C \norm{   M^{i} f }_{(2,0)}+16L^2 \,i\, \|   M^{i-\frac{1}{2}}f  \|_{(2,0)}+C   \frac{\epsilon }{(2i+1)^3}  C_*^{ 2i} (2i)!\\
  &\leq  C \norm{   M^{i} f }_{(2,0)}+C i^2 \|   M^{i-1}f  \|_{(2,0)}+C   \frac{\epsilon }{(2i+1)^3}  C_*^{ 2i} (2i)!,
  \end{aligned}
  \end{equation*}
  the last inequality using the fact that $\|   M^{i-\frac{1}{2}}f  \|_{(2,0)}\leq \|   M^{i}f  \|_{(2,0)}^{1/2}\|   M^{i-1}f  \|_{(2,0)}^{1/2}$.  This with the inductive assumption \eqref{ass:ind} yields
  \begin{equation}\label{asseven}
  \forall\ 1\leq i\leq k-1,\quad		\sup_{t_0\leq t\leq 1}\norm{   M^i (\phi f) }_{(2,0)}  \leq C   \frac{\epsilon }{(2i+1)^3}  C_*^{ 2i} (2i)!
  \end{equation}
  and
  \begin{align*}
   \norm{  M^{k} (\phi f) }_{(2,0)} \leq  C \norm{M^{k} f }_{(2,0)}  +C   \frac{\epsilon }{(2k+1)^3}  C_*^{ 2k} (2k)!.
  \end{align*}
 The proof of Lemma \ref{lem:tec0} is completed.
\end{proof}

\begin{corollary} \label{cor:tec1}
	Suppose the  inductive assumption \eqref{ass:ind} in Proposition \ref{prop:com} holds.  Let $\phi=\phi(v)$ be a given function of $v$ variable satisfying the estimate
 \eqref{upgiven}.
  Then with the notations in Remark \ref{rem:nota},
  \begin{eqnarray*}
  \forall\,(\ell, p)\in\mathbb Z_+^2 \textrm{ with }	\ell+2p\leq 2k-2,\quad \sup_{t_0\leq t\leq 1}\norm{\Lambda^\ell  M^p (\phi f) }_{(2,0)} \leq C   \frac{\epsilon }{(\ell+2p+1)^3}  C_*^{ \ell+2p} (\ell+2p)!.
  \end{eqnarray*}
  Moreover,
   \begin{eqnarray*}
  	\sup_{t_0\leq t\leq 1}\norm{\Lambda^1  M^{k-1} (\phi f) }_{(2,0)} \leq C    \sup_{t_0\leq t\leq 1} \norm{M^{k} f }_{(2,0)} ^{1\over2}\bigg(     \frac{\epsilon }{(2k+1)^3}  C_*^{ 2k-2} (2k-2)!\bigg)^{1\over2}   +     C   \frac{\epsilon  C_*^{ 2k-1} (2k-1)!}{(2k+1)^3} .
       \end{eqnarray*}
\end{corollary}

\begin{proof} We prove the first assertion.
If   $\ell$ is even, then by the assumption   $0\leq \ell+2p\leq 2k-2$  it follows that
\begin{equation*}\label{el}
0\leq \frac{\ell}{2}+p\leq k-1.
\end{equation*}
 As a result,  we use Lemmas    \ref{MM} and  \ref{lem:tec0} to conclude that the following estimate
  \begin{equation}\label{asseven+}
  		\sup_{t_0\leq t\leq 1}\norm{\Lambda^\ell  M^p (\phi f) }_{(2,0)}\leq \sup_{t_0\leq t\leq 1}\norm{   M^{p+\frac{\ell}{2}} (\phi f) }_{(2,0)} \leq C   \frac{\epsilon }{(2p+\ell+1)^3}  C_*^{ 2p+\ell} (2p+\ell)!
  \end{equation}
  holds true for any pair  $(\ell,p)\in\mathbb Z_+^2$   with $\ell$ even and   $0\leq \ell+2p\leq 2k-2$.

  Now we deal with the case when $\ell$ is odd and $\ell+2p\leq 2k-2$,  which implies $\ell+2p\leq 2k-3$ and thus $p+\frac{\ell+1}{2}\leq k-1$.  This enables us to apply again   Lemmas \ref{MM}  and \ref{lem:tec0}  to compute
	\begin{align*}
		\norm{\Lambda^\ell M^p (\phi f) }_{(2,0)}&\leq \norm{  M^{p+\frac{\ell-1}{2}} (\phi f) }_{(2,0)}^{1/2}\norm{  M^{p+\frac{\ell+1}{2}} (\phi f) }_{(2,0)}^{1/2}\\
		&\leq C   \frac{\epsilon }{(2p+\ell+1)^3}  \bigg[C_*^{ 2p+\ell-1} (2p+\ell-1)!C_*^{ 2p+\ell+1} (2p+\ell+1)!\bigg]^{1/2}\\
	& \leq  C   \frac{\epsilon }{(2p+\ell+1)^3}  C_*^{ 2p+\ell-1} (2p+\ell)!,
	\end{align*}
	 which holds for all pair $(\ell, p)\in\mathbb Z_+^2$  with  $\ell$ odd and $ \ell+2p\leq 2k-2$.
This with \eqref{asseven+} yields the first assertion in Corollary \ref{cor:tec1}.

As for the second assertion we make use of  Lemmas \ref{MM} and \ref{lem:tec0} to get
\begin{eqnarray*}
\begin{aligned}
	 \norm{\Lambda^1 M^{k-1} (\phi f) }_{(2,0)}&\leq\norm{ M^{k} (\phi f) }_{(2,0)}^{1/2}\norm{ M^{k-1} (\phi f) }_{(2,0)}^{1/2}\\
 &\leq C \bigg( \norm{M^{k} f }_{(2,0)}  +    \frac{\epsilon }{(2k+1)^3}  C_*^{ 2k} (2k)!\bigg)^{1/2}\bigg(     \frac{\epsilon }{(2k+1)^3}  C_*^{ 2k-2} (2k-2)!\bigg)^{1/2}\\
	&\leq C    \sup_{t_0\leq t\leq 1} \norm{M^{k} f }_{(2,0)} ^{1\over 2}\bigg(     \frac{\epsilon }{(2k+1)^3}  C_*^{ 2k-2} (2k-2)!\bigg)^{1\over 2}   +     C   \frac{\epsilon }{(2k+1)^3}  C_*^{ 2k-1} (2k-1)!.
	\end{aligned}
	\end{eqnarray*}
	The proof of Corollary  \ref{cor:tec1} is completed.
\end{proof}

\begin{lemma}
	[technical lemma]\label{lem:tec2} Suppose the  inductive assumption \eqref{ass:ind} in Proposition \ref{prop:com} holds.  Then
	\begin{equation}\label{diffest}
	\forall\, (\ell,q)\in\mathbb Z_+^2 \textrm{ with }  \ell+2q\leq 2k-2,\quad  \bigg( \int_{t_0}^1\norm{\comi v^{\gamma \over 2}\Lambda^\ell  M^q \mathcal P f}_{(2,0)}^2dt\bigg)^{1\over2}\leq
    \frac{C \epsilon   }{(2q+\ell+1)^3}C_*^{ 2q+\ell} (2q+\ell)!,	 	\end{equation}
and moreover,	
	\begin{multline}\label{diffest+}
	\bigg( \int_{t_0}^1\norm{\comi v^{\gamma\over2}\Lambda^1 M^{k-1} \mathcal P f}_{(2,0)}^2dt\bigg)^{1\over2}\\ \leq C\bigg( \int_{t_0}^1\norm{ \comi v^{\gamma\over2}  \mathcal P M^{k} f}_{(2,0)}^2dt\bigg)^{1\over4}\bigg[  \frac{\epsilon }{(2k-1)^3}C_*^{ 2k-2} (2k-2)!\bigg]^{1\over 2}+C \frac{\epsilon }{(2k+1)^3}C_*^{ 2k-2} (2k-1)!, 	 	\end{multline}
 where $\mathcal P$ stands for any one of the  operators $v\wedge\partial_v, \partial_v$ and $v$.
\end{lemma}

\begin{proof}
	  As a preliminary step we first show that
\begin{equation}\label{copre}
\forall\ j\geq 1,\quad	 \norm{\comi v^{\gamma\over2}  M^j \mathcal P f}_{(2,0)}\leq 	 C \norm{\psi(v,D_v) M^j  f}_{(2,0)}+Cj^2\norm{\psi(v,D_v) M^{j-1}  f}_{(2,0)},
\end{equation}
recalling $\norm{\psi(v,D_v) g}_{(2,0)}$ is defined by \eqref{ma+0}.   Without loss of generality we only prove \eqref{copre} for $\mathcal P=v\wedge \partial_v$, and the other cases can be treated in the same way with simpler argument.
Recall  $M=\Lambda_1^2+\Lambda_2^2$ with $\Lambda_j$ given in \eqref{sqroot}. Then Direct verification shows
\begin{equation}\label{comwedge}
[M, v\wedge \partial_v]=2\sum_{1\leq i\leq 2}\big([\Lambda_i , v]\wedge \partial_v\big)\Lambda_i=2\sum_{1\leq i\leq 2}\Lambda_i\big([\Lambda_i , v]\wedge \partial_v\big),  \textrm{ with } [\Lambda_i , v]=c_i\big ((t-t_0)^{1\over 2}, 0, 0\big ),
\end{equation}
where $c_i,1\leq i\leq 2,$ are constants.
Thus
 \begin{eqnarray*}
 \begin{aligned}
& \comi v^{\gamma/2}   M^j (v\wedge\partial_v) =\comi v^{\gamma/2} (v\wedge\partial_v)
   M^j+2j   \sum_{1\leq i\leq 2}\comi v^{\gamma/2}\Lambda_i\big([\Lambda_i , v]\wedge \partial_v\big)  M^{j-1}\\
   & =\comi v^{\gamma/2} (v\wedge\partial_v)
   M^j+2j \sum_{1\leq i\leq 2} \Lambda_i  \comi v^{\gamma/2}  \big([\Lambda_i , v]\wedge \partial_v\big)  M^{j-1}+2j \sum_{1\leq i\leq 2}  [\comi v^{\gamma/2},\ \Lambda_i ]  \big([\Lambda_i , v]\wedge \partial_v\big)  M^{j-1}.
   \end{aligned}
   \end{eqnarray*}
   Moreover observe that
   \begin{eqnarray*}
   \sum_{1\leq i\leq 2}	\norm{ \Lambda_i  \comi v^{\gamma/2}  \big([\Lambda_i , v]\wedge \partial_v\big)  M^{j-1}f}_{(2,0)}\leq C\norm{M^{1/2}\comi v^{\gamma/2}    \partial_v M^{j-1}f}_{(2,0)}
   \end{eqnarray*}
   due to Lemma \ref{MM}, and that
   \begin{eqnarray*}
   	\sum_{1\leq i\leq 2}	\norm{ [\comi v^{\gamma/2},\ \Lambda_i ]  \big([\Lambda_i , v]\wedge \partial_v\big)  M^{j-1}f}_{(2,0)}\leq C\norm{M^{1/2} M^{j-1}f}_{(2,0)},
   \end{eqnarray*}
   which holds because it follows from the fact $\gamma\leq 1$ that
   \begin{equation}\label{uppforvgamma}
\forall\ 1\leq \abs\beta\leq 2,\quad    	|\partial_v^\beta \comi v^{\gamma/2}|\leq C.
   \end{equation}
 This  yields
   \begin{equation}\label{ueti}
   \begin{aligned}
    	&\norm{\comi v^{\gamma\over2}  M^j (v\wedge\partial_v) f}_{(2,0)}\\
    	&\leq	\norm{\comi v^{\gamma\over2}  (v\wedge\partial_v) M^j f}_{(2,0)} +Cj\norm{M^{1/2} \partial_v\comi v^{\gamma\over2}   M^{j-1} f}_{(2,0)}+Cj\norm{M^{1/2} M^{j-1} f}_{(2,0)} \\
    	&\leq	C \norm{\psi(v,D_v)M^j f}_{(2,0)} +Cj^2\norm{\psi(v,D_v)  M^{j-1} f}_{(2,0)}+Cj\norm{M^{1/2} \partial_v\comi v^{\gamma\over2}   M^{j-1} f}_{(2,0)}.    	\end{aligned}
   \end{equation}
   To estimate the last term on the right hand side of \eqref{ueti}, we   write
\begin{align*}
	 \norm{M^{1/2} \partial_v\comi v^{\gamma\over2}   M^{j-1} f}_{(2,0)}^2 & \leq \norm{M  \partial_v\comi v^{\gamma\over2}   M^{j-1} f}_{(2,0)}\norm{ \partial_v\comi v^{\gamma\over2}   M^{j-1} f}_{(2,0)}\\
	 &\leq C \norm{M  \partial_v\comi v^{\gamma\over2}   M^{j-1} f}_{(2,0)}\norm{ \psi(v,D_v)  M^{j-1} f}_{(2,0)},
\end{align*}
and moreover,  by direct computation and using   \eqref{uppforvgamma},
\begin{align*}
		  & \norm{M  \partial_v\comi v^{\gamma\over2}   M^{j-1} f}_{(2,0)}=\norm{   \partial_v\comi v^{\gamma\over2} M  M^{j-1} f+\partial_v [M,\   \comi v^{\gamma\over2}]   M^{j-1} f}_{(2,0)}\\
		  &  \leq  \norm{  \partial_v\comi v^{\gamma\over2}   M^{j} f}_{(2,0)}+\norm{ [M, \comi v^{\gamma\over2}]   \partial_v M^{j-1} f}_{(2,0)}+\norm{ [ \partial_v, \, [M, \comi v^{\gamma\over2}] \,]  M^{j-1} f}_{(2,0)}\\
		  &  \leq  \norm{  \partial_v\comi v^{\gamma\over2}   M^{j} f}_{(2,0)}+C\norm{ M M^{j-1} f}_{(2,0)}+C\norm{ M^{1/2} M^{j-1} f}_{(2,0)} \\
		  &  \leq C \norm{  \psi(v,D_v)  M^{j} f}_{(2,0)} +C\norm{ \psi(v,D_v)  M^{j-1} f}_{(2,0)}.
		\end{align*}
Thus, combining the above estimates yields \begin{eqnarray*}
	\norm{M^{1/2} \partial_v\comi v^{\gamma\over2}   M^{j-1} f}_{(2,0)}  \leq C \norm{  \psi(v,D_v)   M^{j} f}_{(2,0)}^{1/2}\norm{  \psi(v,D_v)  M^{j-1} f}_{(2,0)}^{1/2} +C\norm{  \psi(v,D_v)   M^{j-1} f}_{(2,0)}.
\end{eqnarray*}
Substituting the above inequality into \eqref{ueti} yields \eqref{copre} for $\mathcal P=v\wedge\partial_v.$ The treatment for $\mathcal P= \partial_v $ or $\mathcal P=v $ is similar and simpler, so we omit it for brevity.

Next  we will proceed    to prove \eqref{diffest} and \eqref{diffest+}.  The first two steps are devoted to proving \eqref{diffest} by induction on $\ell$. and the last one to proving \eqref{diffest+}.

{\bf (i)  Initial step}.  This step and the next one are devoted to proving \eqref{diffest} by induction on $\ell.$  For any integer $q\in\mathbb Z_+$ with   $2q\leq 2k-2$, we have $q\leq k-1$ and thus  using \eqref{copre}
and the inductive assumption \eqref{ass:ind} yields
\begin{align*}
&	\bigg( \int_{t_0}^1\norm{\comi v^{\gamma\over2}  M^q \mathcal P f}_{(2,0)}^2dt\bigg)^{1\over2}  \leq  C 	 \bigg(\int_{t_0}^1\Big(\norm{\psi(v,D_v)M^q  f}_{(2,0)}^2 +q^4 	  \norm{\psi(v,D_v)  M^{q-1}  f}_{(2,0)}^2\Big)dt\bigg)^{1\over2}\\
	  & \leq  C     \frac{\epsilon }{(2q+1)^3}  C_*^{ 2q} (2q)!+Cq^2    \frac{\epsilon }{(2q+1)^3}  C_*^{ 2q-2} (2q-2)! \leq   C     \frac{\epsilon }{(2q+1)^3}  C_*^{ 2q} (2q)!.
\end{align*}
We have proven the validity of \eqref{diffest} for  any pair $(\ell, q)\in\mathbb Z_+^2$ with $\ell=0$ and $\ell+2q=2q\leq 2k-2$.

 {\bf (ii) Inductive step}.  Let $\ell\geq 1$ be a given integer, and suppose the following estimate
	\begin{equation}\label{N}
	\bigg( \int_{t_0}^1\norm{\comi v^{\gamma/2}\Lambda^N  M^q \mathcal P f}_{(2,0)}^2dt\bigg)^{1/2}\leq  C     \frac{\epsilon }{(N+2q+1)^3}  C_*^{N+ 2q} (N+2q)!
	 	\end{equation}
   holds true for any  pair    $(N,q)\in\mathbb Z_+^2$   with  $N\leq \ell-1$     and $N+2q\leq 2k-2$.  We will show in this step that  the above estimate \eqref{N} still holds for  any  pair   $(N,q)\in \mathbb Z_+^2$  with  $N=\ell$ and $N+2q=\ell+2q\leq 2k-2$.
  In the following discussion, let  $q$ be any integer   satisfying that  $\ell+2q\leq 2k-2$ with $\ell\geq 1$ given.

   We    first consider the case when  $\ell$ is odd.  Then  the assumption $\ell+2p\leq 2k-2$ implies
    \begin{equation}\label{2k3}
    	\ell+2p\leq 2k-3.
    \end{equation}
       Observe $[\Lambda, \comi v^{\gamma/2}]$ is just the first order derivatives of $\comi v^{\gamma/2}$. Then using \eqref{uppforvgamma} gives
       \begin{multline}\label{eqeti}
\bigg(	 \int_{t_0}^1\norm{\comi v^{\gamma/2}\Lambda^\ell  M^q \mathcal P f}_{(2,0)}^2dt\bigg)^{1/2}\\
	 \leq  \bigg( \int_{t_0}^1\norm{\Lambda^1\comi v^{\gamma/2}\Lambda^{\ell-1}  M^q \mathcal P f}_{(2,0)}^2dt\bigg)^{1/2}+C\bigg(\int_{t_0}^1\norm{    \Lambda^{\ell-1}  M^q \mathcal P f}_{(2,0)}^2dt\bigg)^{1/2}.
\end{multline}
Moreover for the first term on the right hand side,
\begin{equation}\label{leone}
\begin{aligned}
	& \bigg(\int_{t_0}^1\norm{\Lambda^1\comi v^{\gamma/2}\Lambda^{\ell-1}  M^q \mathcal P f}_{(2,0)}^2dt\bigg)^{1/2}\\
	& \leq  \bigg(\int_{t_0}^1\norm{M\comi v^{\gamma/2}\Lambda^{\ell-1}  M^q \mathcal P f}_{(2,0)}\norm{\comi v^{\gamma/2}\Lambda^{\ell-1}  M^q \mathcal P f}_{(2,0)}dt\bigg)^{1/2}\\
	& \leq C \bigg(\int_{t_0}^1\norm{ \comi v^{\gamma/2}\Lambda^{\ell-1}  M^{q+1} \mathcal P f}_{(2,0)}^2dt\bigg)^{1/4}\bigg(\int_{t_0}^1\norm{\comi v^{\gamma/2}\Lambda^{\ell-1}  M^q \mathcal P f}_{(2,0)}^2dt\bigg)^{1/4}\\
	&\quad +\bigg( \int_{t_0}^1\norm{\comi v^{\gamma/2}\Lambda^{\ell-1}  M^q \mathcal P f}_{(2,0)}^2dt\bigg)^{1/2},
\end{aligned}
\end{equation}
where in the last inequality  we have used the fact that
\begin{eqnarray*}
	\norm{[M,\ \comi v^{\gamma/2}] g}_{(2,0)}\leq C \norm{M^{1/2}g}_{(2,0)}\leq C \inner{\norm{M g}_{(2,0)} +C \norm{ g}_{(2,0)} }
\end{eqnarray*}
due to \eqref{uppforvgamma}. As a result, combining the above inequalities gives
\begin{equation}\label{oddintell}
	\begin{aligned}
	& \bigg(\int_{t_0}^1\norm{ \comi v^{\gamma/2}\Lambda^{\ell}  M^q \mathcal P f}_{(2,0)}^2dt \bigg)^{1/2}\leq C \bigg(\int_{t_0}^1\norm{\comi v^{\gamma/2}\Lambda^{\ell-1}  M^q \mathcal P f}_{(2,0)}^2dt\bigg)^{1/2}\\
	&\quad \qquad\qquad+ C \bigg(\int_{t_0}^1\norm{ \comi v^{\gamma/2}\Lambda^{\ell-1}  M^{q+1}\mathcal P  f}_{(2,0)}^2dt\bigg)^{1\over4}\bigg(\int_{t_0}^1\norm{\comi v^{\gamma/2}\Lambda^{\ell-1}  M^q \mathcal P f}_{(2,0)}^2dt\bigg)^{1\over4}\\
	& \leq C \frac{\epsilon }{(2q+\ell)^3}C_*^{ 2q+\ell-1} (2q+\ell-1)! \\
	&\quad 	+ C\bigg[\frac{\epsilon }{(2q+\ell+2)^3}  C_*^{ 2q+\ell+1} (2q+\ell+1)!\bigg]^{1/2}\bigg[\frac{\epsilon }{(2q+\ell)^3}C_*^{ 2q+\ell-1} (2q+\ell-1)! \bigg]^{1/2} 	\\
	&\leq C \frac{\epsilon }{(2q+\ell+1)^3}C_*^{ 2q+\ell} (2q+\ell)!,
	\end{aligned}
\end{equation}
   the second inequality    using  the inductive assumption \eqref{N}  for $N=\ell-1$ by observing \eqref{2k3}.
 We have proven that \eqref{N}  holds true for any pair $(N,q)$ with  $N=\ell$ odd and   $N+2q\leq 2k-2$.

Next we deal with the case when $\ell$ is even.  Similarly as in \eqref{eqeti} and \eqref{leone}, we have
   \begin{align*}
&\bigg(\int_{t_0}^1\norm{ \comi v^{\gamma/2}\Lambda^{\ell}  M^q \mathcal P f}_{(2,0)}^2dt\bigg)^{1/2}\\
&\leq \bigg(\int_{t_0}^1\norm{\Lambda^2 \comi v^{\gamma/2}\Lambda^{\ell-2}  M^q \mathcal P f}_{(2,0)}^2dt\bigg)^{1/2}+C\bigg(\int_{t_0}^1\norm{\Lambda^1 \Lambda^{\ell-2}  M^q \mathcal P f}_{(2,0)}^2dt\bigg)^{1/2}\\
&\leq C\bigg(\int_{t_0}^1\norm{ \comi v^{\gamma\over2} \Lambda^{\ell-2}  M^{q+1} \mathcal P f}_{(2,0)}^2dt\bigg)^{1/2}	 +C\bigg(\int_{t_0}^1\norm{\comi v^{\gamma\over2}\Lambda^{\ell-2}  M^q \mathcal P f}_{(2,0)}^2dt\bigg)^{1/2}\\
&\leq C\frac{\epsilon }{(2q+\ell+1)^3}C_*^{ 2q+\ell} (2q+\ell)!,	
\end{align*}
the last line using again the inductive assumption \eqref{N} for $N=\ell-2$ since $\ell-2+2(q+1)\leq \ell+2q\leq 2k-2.$  Combining the above estimate for even integer $\ell$  with the previous \eqref{oddintell} for odd integer $\ell$, we conclude  \eqref{N}  holds true for  any  pair    $(N,q)\in\mathbb Z_+^2$  with  $N=\ell$ and $N+2q=\ell+2q\leq 2k-2$.  Thus the first assertion \eqref{diffest} follows.

{\bf (iii) The remaining case of  \boldmath $\ell=1$ and $q=k-1$}.  It remains to prove the second assertion \eqref{diffest+}. Applying \eqref{eqeti} and \eqref{leone} for $\ell=1$ and $q=k-1$, we have
\begin{align*}
	&\bigg( \int_{t_0}^1\norm{  v^{\gamma\over2}\Lambda^1  M^{k-1}\mathcal P f}_{(2,0)}^2dt\bigg)^{1\over2} \leq \bigg( \int_{t_0}^1\norm{\Lambda^1\comi v^{\gamma\over2}  M^{k-1}\mathcal P f}_{(2,0)}^2dt\bigg)^{1\over2}+C\bigg( \int_{t_0}^1\norm{   M^{k-1}\mathcal P f}_{(2,0)}^2dt\bigg)^{1\over2} \\
	& \leq   C\bigg( \int_{t_0}^1\norm{ \comi v^{\gamma\over2}  M^{k} \mathcal P f}_{(2,0)}^2dt\bigg)^{1\over4}\bigg( \int_{t_0}^1\norm{ \comi v^{\gamma\over2}  M^{k-1} \mathcal P f}_{(2,0)}^2dt\bigg)^{1\over4}+C \bigg( \int_{t_0}^1\norm{  M^{k-1} \mathcal P f}_{(2,0)}^2dt\bigg)^{1\over2} \\
	& \leq   C\bigg( \int_{t_0}^1\norm{ \comi v^{\gamma\over2}  M^{k} \mathcal P f}_{(2,0)}^2dt\bigg)^{1\over4}\bigg[  \frac{\epsilon }{(2k-1)^3}C_*^{ 2k-2} (2k-2)!\bigg]^{1\over 2}+C \frac{\epsilon }{(2k-1)^3}C_*^{ 2k-2} (2k-2)! \\
	& \leq   C\bigg( \int_{t_0}^1\norm{\psi(v,D_v)  M^{k} f}_{(2,0)}^2dt\bigg)^{1\over4}\bigg[  \frac{\epsilon }{(2k-1)^3}C_*^{ 2k-2} (2k-2)!\bigg]^{1\over 2}+C \frac{\epsilon }{(2k+1)^3}C_*^{ 2k-2} (2k-1)!,
\end{align*}
where  in the third  inequality we have used \eqref{diffest} and the last line follows by combining \eqref{copre} with \eqref{diffest}.      Then we   obtain  the second assertion   \eqref{diffest+}, completing the proof of Lemma \ref{lem:tec2}.
 \end{proof}

\begin{lemma}
	[Commutator between $M^k$ and $L_1$] \label{lem:L1} Suppose the  inductive assumption \eqref{ass:ind} in Proposition \ref{prop:com} holds.  Then
	\begin{multline*}
	\Big|	\int_{t_0}^1 \Big(M^kL_1(f, f)-L_1\big(f, M^k f\big), M^kf\Big)_{(2,0)}dt \Big|\\
	\leq \inner{ \delta +\epsilon C}\bigg[ \Big(\sup_{t_0\leq t\leq 1} \norm{M^{k} f }_{(2,0)}\Big)^2+  \int_{t_0}^1\norm{ \psi(v,D_v)  M^{k}   f}_{(2,0)}^2dt   \bigg]
	  +  C_\delta \bigg[ \frac{\epsilon^2 }{(2k+1)^3}  C_*^{ 2k} (2k)!   \bigg]^2,
	\end{multline*}
	where $\delta>0$ is an arbitrarily small constant, and $C_\delta$ is a constant depending  on $\delta$.
\end{lemma}

\begin{proof}
	 In view of the representation of $L_1$ given in  Proposition \ref{proprep} we may write
\begin{eqnarray*}
M^kL_1(f, f)-L_1\big(f, M^k f\big)=\sum_{1\leq j\leq 3} S_{j}f,
\end{eqnarray*}
where
\begin{equation}\label{ss}
\left\{
\begin{aligned}
S_{1}&=\frac{1}{2}\big[M^k, \ \inner{v\wedge \partial_v} \cdot a_f
\inner{v\wedge \partial_v}\big],
\\
S_{2}&=\sum_{1\leq i, j\leq 3} \big[M^k, \ \partial_{v_i} M_{i,j,f} \partial_{v_j}\big],
\\
S_{3}&=\frac{1}{2}\big[M^k, \ \big(\partial_v\wedge {\boldsymbol{B}_f \big)}
\cdot \inner{v\wedge \partial_v}\big] -\frac{1}{2}\big[M^k, \ \inner{v\wedge \partial_v} \cdot \big( {\boldsymbol{B}_f} \wedge \partial_v\big)\big],
\end{aligned}
\right.
\end{equation}
with $a_f, \boldsymbol{B}_f$ and $M_{i,j,f}$ defined in \eqref{ag}-\eqref{M}.

 We split   $S_{1}$ as
\begin{equation}\label{s12s11}
\begin{aligned}
S_{1}&
=\frac{1}{2}\big[M^k, \inner{v\wedge \partial_v}\big]\cdot a_f
\inner{v\wedge \partial_v}+
\frac{1}{2} \inner{v\wedge \partial_v}\cdot\big[M^k, a_f\inner{v\wedge \partial_v}\big]
\\
&:= S_{1,1}+S_{1,2}.
\end{aligned}
\end{equation}
We first deal with   $S_{1,2}$  and conclude that, for any $\delta>0$,
	\begin{multline}\label{fs12}
	  \int_{t_0}^1\big|\big(S_{1,2}f,\ M^kf\big)_{(2,0)}\big|dt\\
	   \leq  \inner{ \delta +\epsilon C}\bigg[\Big( \sup_{t_0\leq t\leq 1} \norm{M^{k} f }_{(2,0)}\Big)^2+  \int_{t_0}^1\norm{ \psi(v,D_v)  M^{k}   f}_{(2,0)}^2dt   \bigg]+  C_\delta \bigg[ \frac{\epsilon^2 }{(2k+1)^3}  C_*^{ 2k} (2k)!  \bigg]^2.
\end{multline}
To prove \eqref{fs12},
 we use the fact that
 \begin{align*}
 	\inner{v\wedge \partial_v}\cdot\big[M^k, a_f\inner{v\wedge \partial_v}\big]&=\inner{v\wedge \partial_v}\cdot\big[M^k, a_f   \big] \inner{v\wedge \partial_v}+\inner{v\wedge \partial_v}\cdot a_f \big[M^k,   v \wedge \partial_v \big],
 \end{align*}
to get
 \begin{equation}\label{s12e}
\begin{aligned}
	& \int_{t_0}^1\big|\big(S_{1,2}f,\ M^kf\big)_{(2,0)}\big|dt =\frac12\int_{t_0}^1\big|\big(\inner{v\wedge \partial_v}\cdot\big[M^k, a_f\inner{v\wedge \partial_v}\big]f,\ M^kf\big)_{(2,0)}\big|dt\\
	&  \leq \bigg(\int_{t_0}^1 \norm{ \comi v^{-{\gamma\over2}}\big[M^k, a_f \big] \inner{v\wedge \partial_v}f }_{(2,0)}^2dt\bigg)^{1\over2}\bigg(\int_{t_0}^1\norm{\comi v^{\gamma \over 2}\inner{v\wedge \partial_v}M^kf}_{(2,0)}^2dt\bigg)^{1\over2}\\
	&\quad+\bigg(\int_{t_0}^1 \norm{ \comi v^{-{\gamma\over2}}a_f  \big[M^k, v \wedge \partial_v\big] f }_{(2,0)}^2dt\bigg)^{1\over2}\bigg(\int_{t_0}^1\norm{\comi v^{\gamma \over 2}\inner{v\wedge \partial_v}M^kf}_{(2,0)}^2dt\bigg)^{1\over2}.
	 \end{aligned}
\end{equation}
We  deal with the first  term on the right hand side of  \eqref{s12e}.    In view of \eqref{ag}, we can verify directly
\begin{equation}\label{uppaf}
\forall \ v\in\mathbb R^3,\quad |  a_f(v)|\leq C \comi v^\gamma   \norm{f}_{L_v^2}\ \textrm{ and }\ 	| \Lambda^\ell M^p a_f(v)|\leq C \comi v^\gamma   \norm{\Lambda^\ell M^p\big(\sqrt \mu f\big)}_{L_v^2},
\end{equation}
which,  with Lemma \ref{lem: leibniz} as well as Remark  \ref{rem:nota}, implies
\begin{equation}\label{srt}
\begin{aligned}
	&  \norm{ \comi v^{-{\gamma/2}}\big[M^k, a_f \big] \inner{v\wedge \partial_v}f }_{(2,0)}\\
	&\leq  \sum_{j=1}^{2k} \sum_{\stackrel{ \ell+2p=j}{\ell+2q=2k-j}}c^{k, j}_{\ell, p, q}      \norm{ \Lambda^\ell  M^p\big( \sqrt \mu f \big)}_{(2,0)}        \norm{ \comi v^{\gamma/2}  \Lambda^{\ell } M^q  \inner{v\wedge \partial_v}f }_{(2,0)}.
\end{aligned}
\end{equation}
Thus
\begin{equation}\label{j1j4}
\begin{aligned}
	&\bigg(\int_{t_0}^1 \norm{ \comi v^{-{\gamma\over2}}\big[M^k, a_f \big] \inner{v\wedge \partial_v}f }_{(2,0)}^2dt\bigg)^{1/2}\\
	&\leq  \sum_{j=1}^{2k}\sum_{\stackrel{ \ell+2p=j}{\ell+2q=2k-j}}c^{k, j}_{\ell, p, q}    \sup_{t_0\leq t\leq 1} \norm{\Lambda^\ell  M^p(\sqrt\mu f)}_{(2,0)}\bigg(\int_{t_0}^1 \norm{ \comi v^{\gamma/2}  \Lambda^{\ell } M^q \big(\inner{v\wedge \partial_v}f\big)}_{(2,0)}^2dt\bigg)^{1/2}  \\
	&:=\sum_{1\leq i\leq 4}J_i
	\end{aligned}
	\end{equation}
	with
	\begin{align*}
		J_1&=\sum_{j=2}^{2k-2}\sum_{\stackrel{ \ell+2p=j}{\ell+2q=2k-j}}c^{k, j}_{\ell, p, q}    \sup_{t_0\leq t\leq 1} \norm{\Lambda^\ell  M^p(\sqrt\mu f)}_{(2,0)}\bigg(\int_{t_0}^1 \norm{ \comi v^{\gamma/2}  \Lambda^{\ell } M^q \big(\inner{v\wedge \partial_v}f\big)}_{(2,0)}^2dt\bigg)^{1/2},\\
		J_2&= \sum_{\stackrel{ \ell+2p=2k}{\ell+2q=0}}c^{k, 2k}_{\ell, p, q}    \sup_{t_0\leq t\leq 1} \norm{  \Lambda^\ell M^p(\sqrt\mu f)}_{(2,0)}\bigg(\int_{t_0}^1 \norm{ \comi v^{\gamma/2}  \Lambda^\ell M^q   \inner{v\wedge \partial_v}f }_{(2,0)}^2dt\bigg)^{1/2},\\
		J_3&= \sum_{\stackrel{ \ell+2p=2k-1}{\ell+2q=1}}c^{k, 2k-1}_{\ell, p, q}    \sup_{t_0\leq t\leq 1} \norm{\Lambda^\ell  M^p(\sqrt\mu f)}_{(2,0)}\bigg(\int_{t_0}^1 \norm{ \comi v^{\gamma/2}  \Lambda^{\ell } M^q \big(\inner{v\wedge \partial_v}f\big)}_{(2,0)}^2dt\bigg)^{1/2},\\
		J_4&= \sum_{\stackrel{ \ell+2p=1}{\ell+2q=2k-1}}c^{k, 1}_{\ell, p, q}    \sup_{t_0\leq t\leq 1} \norm{\Lambda^\ell  M^p(\sqrt\mu f)}_{(2,0)}\bigg(\int_{t_0}^1 \norm{ \comi v^{\gamma/2}  \Lambda^{\ell } M^q \big(\inner{v\wedge \partial_v}f\big)}_{(2,0)}^2dt\bigg)^{1/2}.
	\end{align*}
To derive the upper bounds of $J_i,1\leq i\leq 4,$ we need the following fact:
\begin{equation}
	\label{FouGauss}
	\forall \ \beta\in\mathbb Z_+^3,\quad \norm{\partial_v^\beta \mu^{1/2}}_{L_v^2}+\norm{\partial_v^\beta \mu}_{L_v^2} \leq 2\norm{\eta^\beta e^{-\eta^2/4}}_{L_\eta^2}\leq 16^{\abs\beta+1}\abs\beta!.
\end{equation}
This enables us to use the estimates in Corollary \ref{cor:tec1} and Lemma \ref{lem:tec2}, to compute
\begin{align*}
	J_1&\leq C\sum_{j=2}^{2k-2}\sum_{\stackrel{ \ell+2p=j}{\ell+2q=2k-j}}c^{k, j}_{\ell, p, q}       \frac{\epsilon  }{(\ell+2p+1)^3}  C_*^{ \ell+2p} (\ell+2p)!  \frac{\epsilon}{(2q+\ell+1)^3}C_*^{ 2q+\ell} (2q+\ell)!\\
	&\leq CC_*^{2k}\sum_{j=2}^{2k-2} j! (2k-j!)    \frac{\epsilon^2 }{(j+1)^3(2k-j+1)^3} \sum_{\stackrel{ \ell+2p=j}{\ell+2q=2k-j}}c^{k, j}_{\ell, p, q}\\
		&\leq  CC_*^{2k}(2k)! \sum_{j=2}^{2k-2}   \frac{\epsilon^2 }{(j+1)^3(2k-j+1)^3}\leq C\frac{ \epsilon^2}{(2k+1)^3}  C_*^{2k}(2k)!, \end{align*}
the last line using \eqref{sum}.  By Lemma \ref{lem:tec0} and the inductive assumption \eqref{ass:ind} and the fact $c^{k,2k}_{0,k,0}=1$ in Remark \ref{rem:spec},
\begin{align*}
J_2&=\sum_{\stackrel{ \ell+2p=2k}{\ell+2q=0}}c^{k, 2k}_{\ell, p, q}    \sup_{t_0\leq t\leq 1} \norm{  \Lambda^\ell M^p(\sqrt\mu f)}_{(2,0)}\bigg(\int_{t_0}^1 \norm{ \comi v^{\gamma/2}  \Lambda^\ell M^q   \inner{v\wedge \partial_v}f }_{(2,0)}^2dt\bigg)^{1/2}\\
&=  c^{k, 2k}_{0, k, 0}    \sup_{t_0\leq t\leq 1} \norm{   M^k(\sqrt\mu f)}_{(2,0)}\bigg(\int_{t_0}^1 \norm{ \comi v^{\gamma/2}     \inner{v\wedge \partial_v}f }_{(2,0)}^2dt\bigg)^{1/2}\\
&\leq  C     \bigg[ \norm{M^{k} f }_{(2,0)}  +C   \frac{\epsilon }{(2k+1)^3}  C_*^{ 2k} (2k)!  \bigg]	 \epsilon\leq C\epsilon \norm{M^{k} f }_{(2,0)}  +C   \frac{\epsilon^2 }{(2k+1)^3}  C_*^{ 2k} (2k)!.
\end{align*}
  Using again the estimates in Corollary \ref{cor:tec1} and Lemma \ref{lem:tec2} and the    fact that  $c^{k, 2k-1}_{1, k-1, 0}=2k$ in  Remark \ref{rem:spec}, we have,  for any $\delta>0,$
\begin{eqnarray*}
	\begin{aligned}
		J_3&= \sum_{\stackrel{ \ell+2p=2k-1}{\ell+2q=1}}c^{k, 2k-1}_{\ell, p, q}    \sup_{t_0\leq t\leq 1} \norm{  \Lambda^\ell M^p(\sqrt\mu f)}_{(2,0)}\bigg(\int_{t_0}^1 \norm{ \comi v^{\gamma/2}  \Lambda^\ell M^q   \inner{v\wedge \partial_v}f }_{(2,0)}^2dt\bigg)^{1/2}\\
		&= c^{k, 2k-1}_{1, k-1, 0}    \sup_{t_0\leq t\leq 1} \norm{  \Lambda^1 M^{k-1}(\sqrt\mu f)}_{(2,0)}\bigg(\int_{t_0}^1 \norm{ \comi v^{\gamma/2}  \Lambda^1   \inner{v\wedge \partial_v}f }_{(2,0)}^2dt\bigg)^{1/2}\\
		&\leq C k \bigg[   \sup_{t_0\leq t\leq 1} \norm{M^{k} f }_{(2,0)} ^{1/2}\bigg(     \frac{\epsilon }{(2k+1)^3}  C_*^{ 2k-2} (2k-2)!\bigg)^{1/2}   +        \frac{\epsilon }{(2k+1)^3}  C_*^{ 2k-1} (2k-1)! \bigg] C_*\epsilon\\	
		&\leq    C    \sup_{t_0\leq t\leq 1} \norm{M^{k} f }_{(2,0)} ^{1/2}\bigg(     \frac{\epsilon^2 }{(2k+1)^3}  C_*^{ 2k} (2k)!\bigg)^{1/2}   +     C   \frac{\epsilon^2 }{(2k+1)^3}  C_*^{ 2k} (2k)! \\
		&\leq  \delta \sup_{t_0\leq t\leq 1} \norm{M^{k} f }_{(2,0)}  +  C_\delta  \frac{\epsilon^2 }{(2k+1)^3}  C_*^{ 2k} (2k)!.
	\end{aligned}
\end{eqnarray*}
Similarly, using Corollary \ref{cor:tec1} and Lemma \ref{lem:tec2} and observing  $c^{k, 1}_{1, 0, k-1}=2k$,
\begin{align*}
		J_4&= \sum_{\stackrel{ \ell+2p=1}{\ell+2q=2k-1}}c^{k, 1}_{\ell, p, q}    \sup_{t_0\leq t\leq 1} \norm{\Lambda^\ell  M^p(\sqrt\mu f)}_{(2,0)}\bigg(\int_{t_0}^1 \norm{ \comi v^{\gamma/2}  \Lambda^{\ell } M^q \big(\inner{v\wedge \partial_v}f\big)}_{(2,0)}^2dt\bigg)^{1/2}\\
		&=   c^{k, 1}_{1, 0, k-1}    \sup_{t_0\leq t\leq 1} \norm{\Lambda^1 (\sqrt\mu f)}_{(2,0)}\bigg(\int_{t_0}^1 \norm{ \comi v^{\gamma/2}  \Lambda^1 M^{k-1}  \inner{v\wedge \partial_v}f}_{(2,0)}^2dt\bigg)^{1/2}\\
		&\leq       Ck\bigg( \int_{t_0}^1\norm{ \comi v^{\gamma\over2} \inner{v\wedge \partial_v} M^{k}   f}_{(2,0)}^2dt\bigg)^{1\over4}\bigg[ \frac{\epsilon }{(2k-1)^3}C_*^{ 2k-2} (2k-2)!\bigg]^{1\over 2}  C_*\epsilon \\
		&\quad +  Ck\bigg[  \frac{\epsilon }{(2k-1)^3}C_*^{ 2k-2} (2k-1)!\bigg]  C_*\epsilon \\
		&\leq   C\bigg( \int_{t_0}^1\norm{ \psi(v,D_v)  M^{k}   f}_{(2,0)}^2dt\bigg)^{1\over4}\bigg[ \frac{\epsilon^2 }{(2k-1)^3}C_*^{ 2k} (2k)!\bigg]^{1\over 2}   +       C \frac{\epsilon^2 }{(2k-1)^3}C_*^{ 2k} (2k)! \\
		&\leq  \delta \bigg( \int_{t_0}^1\norm{ \psi(v,D_v)  M^{k}   f}_{(2,0)}^2dt\bigg)^{1/2}  +  C_\delta  \frac{\epsilon^2 }{(2k+1)^3}  C_*^{ 2k} (2k)!.
\end{align*}
Now substituting the above estimates on $J_1-J_4$ into \eqref{j1j4} yields, for any  $\delta>0,$
\begin{multline}\label{fieti}
	\bigg(\int_{t_0}^1 \norm{ \comi v^{-{\gamma\over2}}\big[M^k, a_f \big] \inner{v\wedge \partial_v}f }_{(2,0)}^2dt\bigg)^{1/2}\\
	\leq  \inner{ \delta +\epsilon C}\sup_{t_0\leq t\leq 1} \norm{M^{k} f }_{(2,0)}+  \delta \bigg( \int_{t_0}^1\norm{ \psi(v,D_v)  M^{k}   f}_{(2,0)}^2dt\bigg)^{1/2}  +  C_\delta  \frac{\epsilon^2 }{(2k+1)^3}  C_*^{ 2k} (2k)!.
\end{multline}
Next we deal with the second term on the right hand  side of \eqref{s12e}.
By \eqref{comwedge},
\begin{align*}
	 \big[M^k, v\wedge\partial_v\big]=kM^{k-1} \big[M, v\wedge\partial_v\big]=2kM^{k-1}\sum_{1\leq i\leq 2}\Lambda_i\big(\big[\Lambda_i, v\big]\wedge\partial_v\big),
\end{align*}
which, with the first estimate in  \eqref{uppaf},  implies
\begin{align*}
	&\bigg(\int_{t_0}^1 \norm{ \comi v^{-{\gamma\over2}}a_f  \big[M^k, v\wedge\partial_v\big] f }_{(2,0)}^2dt\bigg)^{1\over2}\leq C\big(\sup_{t_0\leq t\leq 1}\norm{f(t)}_{(2,0)}\big)\bigg(\int_{t_0}^1 \norm{ \comi v^{ {\gamma\over2}} \big[M^k, v\wedge\partial_v\big]f }_{(2,0)}^2dt\bigg)^{1\over2}\\
	&\leq Ck\big(\sup_{t_0\leq t\leq 1}\norm{f(t)}_{(2,0)}\big) \bigg(\int_{t_0}^1 \norm{ \comi v^{ {\gamma\over2}}\Lambda^1  M^{k-1} \partial_vf }_{(2,0)}^2dt\bigg)^{1\over2}\\
	&\leq C k\epsilon \bigg( \int_{t_0}^1\norm{ \comi v^{\gamma\over2} \partial_v M^{k} f}_{(2,0)}^2dt\bigg)^{1\over4}\bigg[C \frac{\epsilon }{(2k-1)^3}C_*^{ 2k-2} (2k-2)!\bigg]^{1\over 2}+Ck \frac{\epsilon^2 }{(2k-1)^3}C_*^{ 2k-2} (2k-1)!\\
	&\leq \delta \bigg( \int_{t_0}^1\norm{ \psi(v,D_v)  M^{k}   f}_{(2,0)}^2dt\bigg)^{1/2}  +  C_\delta  \frac{\epsilon^2 }{(2k+1)^3}  C_*^{ 2k} (2k)!
\end{align*}
for any $\delta>0,$ where in the third inequality we have used \eqref{diffest+} in Lemma \ref{lem:tec2}.  As a result, we substitute the above estimate and \eqref{fieti} into \eqref{s12e}, to obtain the  estimate \eqref{fs12}.

Next we deal with $S_{1,1}$ in \eqref{s12s11} and write
\begin{eqnarray*}
\begin{aligned}
	S_{1,1}&=\frac{1}{2}\big[M^k, \inner{v\wedge \partial_v}\big]\cdot a_f
\inner{v\wedge \partial_v}=\frac{k}{2}\big[M, \inner{v\wedge \partial_v}\big]\cdot M^{k-1} a_f
\inner{v\wedge \partial_v}\\
&=\frac{k}{2}\big[M, \inner{v\wedge \partial_v}\big]\cdot a_f M^{k-1}
\inner{v\wedge \partial_v} +\frac{k}{2}  \big[M, \inner{v\wedge \partial_v}\big]\cdot \big[M^{k-1}, a_f
\big ]( v\wedge \partial_v).
\end{aligned}
\end{eqnarray*}
This, with the fact that
\begin{eqnarray*}
	\big[M, \inner{v\wedge \partial_v}\big]= 2\sum_{1\leq i\leq 2}\big(\big[\Lambda_i ,  v\big]\wedge \partial_v\big)\Lambda_i
\end{eqnarray*}
in view of \eqref{comwedge},
yields
\begin{equation}\label{laineq}
\begin{aligned}
	 &\int_{t_0}^1\big|\big(S_{1,1}f,\ M^kf\big)_{(2,0)}\big|dt\\
	&\leq C k   \bigg(\int_{t_0}^1\norm{ \comi v^{-\gamma/2}  \Lambda^1 a_f M^{k-1}(v\wedge\partial_v)f}_{(2,0)}^2dt\bigg)^{1\over2}\bigg(\int_{t_0}^1\norm{\comi v^{\gamma/2}  \partial_v M^kf}_{(2,0)}^2dt\bigg)^{1\over2}\\
	&\quad + C k   \bigg(\int_{t_0}^1\norm{ \comi v^{-\gamma/2}  \Lambda^1 \big[ M^{k-1}, a_f\big](v\wedge\partial_v)f}_{(2,0)}^2dt\bigg)^{1\over2}\bigg(\int_{t_0}^1\norm{\comi v^{\gamma/2}  \partial_v M^kf}_{(2,0)}^2dt\bigg)^{1\over2}.
\end{aligned}
\end{equation}
Moreover, writing $\Lambda^1 (a_f g)= a_f \Lambda^1  g+(\Lambda^1  a_f)g $ and  then
using \eqref{uppaf},
\begin{align*}
	&k   \bigg(\int_{t_0}^1\norm{ \comi v^{-\gamma/2}  \Lambda^1 a_f M^{k-1}(v\wedge\partial_v)f}_{(2,0)}^2dt\bigg)^{1\over2}\\
	& \leq C k \big(\sup_{t_0\leq t\leq 1} \norm{f}_{(2,0)}\big)  \bigg(\int_{t_0}^1\norm{ \comi v^{ \gamma/2}  \Lambda^1M^{k-1}(v\wedge\partial_v)f}_{(2,0)}^2dt\bigg)^{1\over2}\\
	 &\quad   +  Ck\big(\sup_{t_0\leq t\leq 1} \norm{\Lambda^1(\sqrt\mu f)}_{(2,0)}\big)\bigg(\int_{t_0}^1\norm{ \comi v^{ \gamma/2}  M^{k-1}(v\wedge\partial_v)f}_{(2,0)}^2dt\bigg)^{1\over2}\\
	 &\leq Ck\epsilon \bigg( \int_{t_0}^1\norm{ \comi v^{\gamma\over2} (v\wedge\partial_v)M^{k} f}_{(2,0)}^2dt\bigg)^{1\over4}\bigg[\frac{\epsilon }{(2k-1)^3}C_*^{ 2k-2} (2k-2)!\bigg]^{1\over 2}\\
	 &\quad +Ck \frac{\epsilon^2 }{(2k-1)^3}C_*^{ 2k-2} (2k-1)! +C C_*k     \frac{ \epsilon^2   }{(2k-1)^3}C_*^{ 2k-2} (2k-2)!\\
	 &\leq C \bigg( \int_{t_0}^1\norm{\psi(v,D_v)M^{k} f}_{(2,0)}^2dt\bigg)^{1\over4}\bigg[\frac{\epsilon^2 }{(2k-1)^3}C_*^{ 2k-2} (2k)!\bigg]^{1\over 2}+C \frac{\epsilon^2 }{(2k+1)^3}C_*^{ 2k-1 } (2k)!\\
	 &\leq \delta   \bigg( \int_{t_0}^1\norm{\psi(v,D_v)M^{k} f}_{(2,0)}^2dt\bigg)^{1\over2}+C_\delta  \frac{\epsilon^2  }{(2k+1)^3}C_*^{ 2k} (2k)!
\end{align*}
the second inequality using the estimates \eqref{diffest}-\eqref{diffest+} in Lemma \ref{lem:tec2} as well as the first estimate in Corollary \ref{cor:tec1}.   As a result,
\begin{multline}\label{klat}
	k   \bigg(\int_{t_0}^1\norm{ \comi v^{-\gamma/2}  \Lambda^1 a_f M^{k-1}(v\wedge\partial_v)f}_{(2,0)}^2dt\bigg)^{1\over2}\bigg(\int_{t_0}^1\norm{\comi v^{\gamma/2}  \partial_v M^kf}_{(2,0)}^2dt\bigg)^{1\over2}\\
	\leq \delta     \int_{t_0}^1\norm{\psi(v,D_v)M^{k} f}_{(2,0)}^2dt +C_\delta \bigg[ \frac{\epsilon^2  }{(2k+1)^3}C_*^{ 2k} (2k)!\bigg]^2.
\end{multline}
We have the upper bound of the first term on the right side of \eqref{laineq}, and  it remains to deal with the second one. Similarly to \eqref{srt} we have
\begin{multline*}
	\norm{ \comi v^{-\gamma/2}  \Lambda^1 \big[ M^{k-1}, a_f\big](v\wedge\partial_v)f}_{(2,0)}\\
	\leq  \sum_{j=1}^{2k-2} \sum_{\stackrel{ \ell+2p=j}{\ell+2q=2k-2-j}}c^{k-1, j}_{\ell, p, q}      \norm{ \Lambda^{\ell+1}  M^p\big( \sqrt \mu f \big)}_{(2,0)}        \norm{ \comi v^{\gamma/2}  \Lambda^{\ell } M^q  \inner{v\wedge \partial_v}f }_{(2,0)}\\
	+   \sum_{j=1}^{2k-2} \sum_{\stackrel{ \ell+2p=j}{\ell+2q=2k-2-j}}c^{k-1, j}_{\ell, p, q}      \norm{ \Lambda^\ell  M^p\big( \sqrt \mu f \big)}_{(2,0)}        \norm{ \comi v^{\gamma/2}  \Lambda^{\ell +1} M^q  \inner{v\wedge \partial_v}f }_{(2,0)}.
\end{multline*}
Then repeating  the argument for treating $J_1$ in  \eqref{j1j4}  we conclude ,
using \ref{cor:tec1} and Lemma \ref{lem:tec2},
\begin{align*}
	&k   \bigg(\int_{t_0}^1\norm{ \comi v^{-\gamma/2}  \Lambda^1 \big[ M^{k-1}, a_f\big](v\wedge\partial_v)f}_{(2,0)}^2dt\bigg)^{1\over2}\\
	&\leq CC_*^{2k-2}k\sum_{j=2}^{2k-2} (j+1)! (2k-2-j!)    \frac{\epsilon^2 }{(j+2)^3(2k-j-1)^3} \sum_{\stackrel{ \ell+2p=j}{\ell+2q=2k-2-j}}c^{k-1, j}_{\ell, p, q}\\
	&\quad+ CC_*^{2k-2}k\sum_{j=2}^{2k-2} j! (2k-1-j!)    \frac{\epsilon^2 }{(j+1)^3(2k-j)^3} \sum_{\stackrel{ \ell+2p=j}{\ell+2q=2k-2-j}}c^{k-1, j}_{\ell, p, q}\\
		&\leq \big(  CC_*^{2k-2}k\big)(2k-2)! \sum_{j=2}^{2k-2}   \frac{\epsilon^2[ j+(2k-j)] }{(j+1)^3(2k-j-1)^3}\leq C\frac{ \epsilon^2}{(2k+1)^3}  C_*^{2k}(2k)!. \end{align*}
 Thus, for any $\delta>0,$
 \begin{multline*}
 	k   \bigg(\int_{t_0}^1\norm{ \comi v^{-\gamma/2}  \Lambda^1 \big[ M^{k-1}, a_f\big](v\wedge\partial_v)f}_{(2,0)}^2dt\bigg)^{1\over2}\bigg(\int_{t_0}^1\norm{\comi v^{\gamma/2}  \partial_v M^kf}_{(2,0)}^2dt\bigg)^{1\over2}\\
 	\leq\delta     \int_{t_0}^1\norm{\psi(v,D_v)M^{k} f}_{(2,0)}^2dt +C_\delta \bigg[ \frac{\epsilon^2  }{(2k+1)^3}C_*^{ 2k} (2k)!\bigg]^2.
 \end{multline*}
 Substituting the above estimate and  \eqref{klat} into \eqref{laineq} we obtain
\begin{align*}
	  \int_{t_0}^1\big|\big(S_{1,1}f,\ M^kf\big)_{(2,0)}\big|dt \leq  \delta    \int_{t_0}^1\norm{ \psi(v,D_v)  M^{k}   f}_{(2,0)}^2dt   +  C_\delta \bigg[ \frac{\epsilon^2  }{(2k+1)^3}  C_*^{ 2k} (2k)!  \bigg]^2.
\end{align*}
This with \eqref{fs12} as well as \eqref{s12s11} yields
\begin{multline}\label{las1et}
	  \int_{t_0}^1\big|\big(S_{1}f,\ M^kf\big)_{(2,0)}\big|dt 	 \\  \leq  \inner{ \delta +\epsilon C}\bigg[ \Big(\sup_{t_0\leq t\leq 1} \norm{M^{k} f }_{(2,0)}\Big)^2+  \int_{t_0}^1\norm{ \psi(v,D_v)  M^{k}   f}_{(2,0)}^2dt   \bigg]
	  +  C_\delta \bigg[ \frac{\epsilon^2 }{(2k+1)^3}  C_*^{ 2k} (2k)!   \bigg]^2.
\end{multline}
By the definition of $M_{i,j,f}$ and $\boldsymbol{B}_f$ given in \eqref{ag} and \eqref{M}, we can verify that similar to \eqref{uppaf}, the following estimates
\begin{eqnarray*}
	  |  \boldsymbol{B}_f(v)| \leq C \comi v^\gamma   \norm{f}_{L_v^2}\ \textrm{ and }\ 	| \Lambda^\ell M^p \boldsymbol{B}_f(v)|\leq C \comi v^\gamma   \norm{\Lambda^\ell M^p\big(v\sqrt \mu f\big)}_{L_v^2}
\end{eqnarray*}
and
\begin{eqnarray*}
		  | M_{i,j,f}(v)|\leq C \comi v^\gamma   \norm{f}_{L_v^2}\ \textrm{ and }\ 	| \Lambda^\ell M^p M_{i,j,f}(v)|\leq C \comi v^\gamma   \norm{\Lambda^\ell M^p\big((\delta_{i,j}\abs v^2-v_iv_j) \sqrt \mu f\big)}_{L_v^2}
\end{eqnarray*}
hold true for any $ v\in\mathbb R^3$.  This with \eqref{ss} enables us to repeat the above argument for estimating $S_1$ with slight modifications, to conclude that \eqref{las1et} still holds true with $S_1$ replaced by $S_j, 2\leq j\leq 3.$ The proof of Lemma \ref{lem:L1} is thus completed.
 \end{proof}

\begin{lemma}
	[Commutator between $M^k$ and $L_j$, $2\leq j\leq 6$] \label{lem: L26} Suppose the  inductive assumption \eqref{ass:ind} in Proposition \ref{prop:com} holds. Then
	\begin{multline*}
	\sum_{2\leq j\leq 6}\Big|	\int_{t_0}^1 \Big(M^kL_j(f, f)-L_j\big(f, M^k f\big), M^kf\Big)_{(2,0)}dt \Big|\\
	\leq \inner{ \delta +\epsilon C}\bigg[ \Big(\sup_{t_0\leq t\leq 1} \norm{M^{k} f }_{(2,0)}\Big)^2+  \int_{t_0}^1\norm{ \psi(v,D_v)  M^{k}   f}_{(2,0)}^2dt   \bigg]
	  +  C_\delta \bigg[ \frac{\epsilon^2 }{(2k+1)^3}  C_*^{ 2k} (2k)!    \bigg]^2,
	\end{multline*}
	where $\delta>0$ is an arbitrarily small constant.
\end{lemma}

\begin{proof}
	The argument  is quite similar as that in the proof of  Lemma \ref{lem:L1}. Since there is no additional difficulty,  we omit it for brevity.
\end{proof}

\begin{lemma}
	[Commutator between $M^k$ and $\mathcal L$] \label{lem: Linear} Suppose the  inductive assumption \eqref{ass:ind} in Proposition \ref{prop:com} holds. Then
\begin{eqnarray*}
 \int_{t_0}^1\big|\left([M^k,\,\mathcal{L}]\, f,\, M^k f\right)_{(2,0)}\big|dt
	\leq  \delta     \int_{t_0}^1\norm{ \psi(v,D_v)  M^{k}   f}_{(2,0)}^2dt
	  +  C_\delta \bigg[ \frac{\epsilon }{(2k+1)^3}  C_*^{ 2k-1} (2k)!  \bigg]^2.
\end{eqnarray*}	where $\delta>0$ is an arbitrarily small constant.
\end{lemma}

\begin{proof}
Observe $M^j \mu=-(t-t_0)^j\partial_{v_1}^{2j}\mu$, and  thus  it follows from   Lemma \ref{MM} and
	   \eqref{FouGauss} that
		\begin{equation}\label{esformu}
		\forall\ (\ell, p)\in\mathbb Z_+^2,\quad  \norm{\Lambda^\ell  M^p \mu }_{L_v^2}\leq \norm{   M^{\frac{\ell}{2}+p} \mu }_{L_v^2}  \leq  16^{ \ell+2p+1} (\ell+2p)!.
	\end{equation}
	Then following the argument in the proof of Lemma \ref{lem:L1} and using  \eqref{esformu} instead of  Corollary \ref{cor:tec1}, we conclude by direct computation  that
	\begin{multline*}
			\sum_{1\leq j\leq 6}\Big|	\int_{t_0}^1 \Big(M^kL_j(\sqrt\mu, f)-L_j\big(\sqrt\mu, M^k f\big), M^kf\Big)_{(2,0)}dt \Big| \\
	\leq  \delta   \int_{t_0}^1\norm{ \psi(v,D_v)  M^{k}   f}_{(2,0)}^2dt
	  +  C_\delta \bigg[ \frac{\epsilon }{(2k+1)^3}  C_*^{ 2k-1} (2k)!    \bigg]^2.
	\end{multline*}
	Similarly,
	\begin{multline*}
			\sum_{1\leq j\leq 6}\Big|	\int_{t_0}^1 \Big(M^kL_j(f, \sqrt\mu)-L_j\big(M^k f, \sqrt\mu\big), M^kf\Big)_{(2,0)}dt \Big| \\
	\leq  \delta   \int_{t_0}^1\norm{ \psi(v,D_v)  M^{k}   f}_{(2,0)}^2dt
	  +  C_\delta \bigg[ \frac{\epsilon }{(2k+1)^3}  C_*^{ 2k-1} (2k)!    \bigg]^2.
	\end{multline*}
	As a result,  observing
	\begin{align*}
		-[M^k,\,\mathcal{L}] f&=\Big(M^k\Gamma(\sqrt\mu, f)-\Gamma(\sqrt\mu, M^kf)\Big)+\Big(M^k\Gamma(f, \sqrt\mu)-\Gamma(M^kf, \sqrt\mu)\Big)\\
		&=\sum_{1\leq j\leq 6}\Big(M^kL_j(\sqrt\mu, f)-L_j(\sqrt\mu, M^kf)\Big)+\sum_{1\leq j\leq 6}\Big(M^kL_j(f, \sqrt\mu)-L_j(M^kf, \sqrt\mu)\Big)
	\end{align*}
	due to  Proposition \ref{proprep}, we  obtain the assertion in Lemma \ref{lem: Linear}, completing the proof.
\end{proof}

\begin{proof}
	[Proof of Proposition \ref{prop:com}] Combining the estimates in Lemmas \ref{lem:L1} and \ref{lem: L26} with the presentation of $\Gamma(f,f)$ given in Proposition \ref{proprep}, we conclude
	\begin{multline*}
	\int_{t_0}^1\big|\left(M^k\Gamma(f, f)-\Gamma(f, M^kf), M^k f\right)_{(2,0)}\big|  \\
	\leq \inner{ \delta +\epsilon C}\bigg[ \Big(\sup_{t_0\leq t\leq 1} \norm{M^{k} f }_{(2,0)}\Big)^2+  \int_{t_0}^1\norm{ \psi(v,D_v)  M^{k}   f}_{(2,0)}^2dt   \bigg] 	  +  C_\delta \bigg[ \frac{\epsilon^2 }{(2k+1)^3}  C_*^{ 2k} (2k)!    \bigg]^2.
\end{multline*}
This with Lemma \ref{lem: Linear} yields the assertion in
  Proposition \ref{prop:com}. The proof  is completed. 	
\end{proof}

\section{Analytic   regularization effect of weak solutions}

In this part we complete  the proof of Theorem \ref{thm:main}.  We begin with the existence and Gevrey regularity of weak solutions to \eqref{cau}, and then improve the Gevrey regularity to analyticity.
Recall the Gevrey class, denoted by $G^\sigma$,   consists of all $C^\infty$ smooth functions $g$ such that
 \begin{eqnarray*}
 \exists\ C>0,\ 	\forall \ \alpha,\beta\in\mathbb Z_+^3,\quad \norm{ {\partial_x^\alpha\partial_v^\beta g}}_{L^2}
 \leq   C^{|\alpha|+|\beta|+1}[ (|\alpha|+|\beta|) !]^{\sigma}.
 \end{eqnarray*}

   \begin{theorem}\label{thm: Gevrey}
   Under the same assumption as in Theorem \ref{thm:main},  the Cauchy problem \eqref{cau} admits
   a unique global-in-time  solution $f $ satisfying that
  \begin{eqnarray*}
  	 \sup_{t>0}   \norm{  f(t) }_{L^2} +\bigg(\int_0^{+\infty}     \norm{\psi(v,D_v)   f (t) }_{L^2} ^2ds\bigg)^{1\over2}
 \leq   C_0 \eps_0
  \end{eqnarray*}
   for some
     constant $C_0\geq 1$ depending only the initial data, where $\eps_0$ is the small sufficiently number given in \eqref{123}.  Moreover the following estimate
    \begin{equation}\label{thm.gest}
   \sup_{t>0} \tilde t^{\frac{3}{2} (\abs{\alpha}+|\beta|)}   \norm{ {\partial_x^\alpha\partial_v^\beta  f(t)}}_{L^2} +\bigg(\int_0^{+\infty}  \tilde t^{3(\abs{\alpha}+|\beta|)}   \norm{\psi(v,D_v) {\partial_x^\alpha\partial_v^\beta f (t)}}_{L^2} ^2ds\bigg)^{1\over2}
 \leq   C_0^{|\alpha|+|\beta|+1} (|\alpha|+|\beta|) ! ^{\frac{3}{2}}
 \end{equation}
holds true  for all  $\alpha, \beta\in\mathbb Z_+^3$,  recalling $\tilde t:=\min\big\{t,  1\big\}$.
 \end{theorem}

\begin{proof}
	[Sketch of the proof of Theorem \ref{thm: Gevrey}] The global-in-time existence and uniqueness of mild solutions $f(t,x,v)$ in the low-regularity space $L^1_m L^2_v\subset L^2$ was   established by \cite[Theorem 2.1]{MR4230064}.  Furthermore we may  follow the presentation in  \cite{MR4356815} with necessary modifications, to conclude
	 the global Gevrey smoothing effect of such low-regularity solutions in the sense that $f(t,\cdot,\cdot)\in G^{3/2}$ for any $t>0$, that is,  the quantitative estimate \eqref{thm.gest} is satisfied globally in time.
	
	 Observe the proof in \cite{MR4356815}  relies on the same trilinear and coercivity estimates  for Boltzmann collision operator as  that in Proposition \ref{prop: tricoe},  the pseudo-differential  calculus for the symbol
	 \begin{eqnarray*}
	 	\comi v^{\gamma} \inner{1+v^2+\eta^2+\inner{v\wedge\eta}^2}^s,\quad  s\in]0,1[,
	 \end{eqnarray*}
	 which corresponds to $s=1$ in the case of Landau equations.  Thus the estimate \eqref{thm.gest} will follow without any   additional difficulty by applying the same strategy as in   \cite{MR4356815}, and we refer to interested reads to \cite{caoliu} for the detailed derivation of \eqref{thm.gest}.
\end{proof}

\begin{proof}
	[Completing the proof of Theorem \ref{thm:main}] Let    $f $   be the unique global-in-time solution to \eqref{cau} constructed in Theorem \ref{thm: Gevrey} such that    the quantitative estimate \eqref{thm.gest} is fulfilled.     In view of \eqref{thm.gest},  it follows that $f\in L^\infty\inner{[t_0, +\infty[; H^{+\infty}}$ for any $0<t_0\leq 1/2$ and moreover
   \begin{eqnarray*}
\sup_{t_0\leq t\leq 1}	\norm{f(t)}_{(2,0)} + \bigg(\int_{t_0}^1 \norm{\psi(v,D_v)  f(t)}_{(2,0)}^2dt\bigg)^{1\over2} \leq \big( 2 C_0/t_0\big)^3 \eps_0
\end{eqnarray*}
and
\begin{equation*}
 \forall\ \alpha,\beta\in\mathbb Z_+^3, \quad	 \sup_{t_0\leq t\leq 1}   \norm{ {\partial_x^\alpha\partial_v^\beta  f(t)}}_{(2,0)} +\bigg (\int_{t_0}^{1}   \norm{\psi(v,D_v) {\partial_x^\alpha\partial_v^\beta f (t)}}_{L^2} ^2ds\bigg)^{1\over2}
 <+\infty.
\end{equation*}
Then the assumptions in Theorem \ref{thm:key} are fulfilled with $\epsilon=\big( 2 C_0/t_0\big)^3 \eps_0$,  and thus
we  apply Theorem \ref{thm:key}  to conclude that there exists a constant $C_*\geq 1,$   such that  \begin{equation}\label{fmul}
\forall k\ \in\mathbb Z_+,\quad 	\sup_{t_0\leq t\leq 1} \norm{   M^{k}f(t)}_{(2,0)}+\Big(\int_{t_0}^1 \norm{\psi(v,D_v)    M^{k} f(t)}_{(2, 0)}^2dt\Big)^{1\over2} \leq  \frac{\big( 2 C_0/t_0\big)^3 \eps_0}{(2k+1)^3}  C_*^{2k}  (2k)!.
\end{equation}
Recall $   M$ is a Fourier multiplier with symbol
\begin{eqnarray*}
	  (t-t_0)\eta_1^2+(t-t_0)^2 \eta_1m_1  +\frac{(t-t_0)^3}{3} m_1^2,
\end{eqnarray*}
which   with  \eqref{fmul} and \eqref{elliptic}
implies,  for any $t_0\in \big ]0,  \ 1/2 \big ],$
\begin{equation*}
  \sup_{t\in[t_0,1]}\Big[ (t-t_0)^k\norm{\partial_{v_1}^{2k}f(t)}_{(2,0)}+(t-t_0)^{3k}\norm{\partial_{x_1}^{2k}f(t)}_{(2,0)}\Big] \leq \big( 2 C_0/t_0\big)^3 \eps_0  (C_*/c_0)^{2k}  (2k)!
\end{equation*}
with $c_0$ the constant given in \eqref{elliptic}.  In particular, letting $t=2t_0\in [t_0, 1]$ in the above estimate yields
\begin{equation*}
\forall\ k\in\mathbb Z_+,\ 	\forall\  t_0\in \big ]0,  \ 1/2 \big ], \quad      t_0 ^k\norm{\partial_{v_1}^{2k}f(2t_0)}_{(2,0)}+ t_0 ^{3k}\norm{\partial_{x_1}^{2k}f(2t_0)}_{(2,0)}  \leq t_0^{-3} \big( 2 C_0\big)^3 \eps_0  (C_*/c_0)^{2k}  (2k)!,
\end{equation*}
that is,
\begin{equation}\label{rev1}
	\forall\ k\in\mathbb Z_+,\ \,	\forall\  t \in \big ]0,  \ 1  \big ], \quad    t  ^{k+3}\norm{\partial_{v_1}^{2k}f(t)}_{(2,0)}+ t ^{3k+3}\norm{\partial_{x_1}^{2k}f(t)}_{(2,0)}  \leq  \big( 4C_0\big)^3 \eps_0  (3C_*/c_0)^{2k}  (2k)! .
\end{equation}
By virtue of the last assertion in Theorem \ref{thm:key},  the estimate \eqref{fmul} also holds  with $   M$ replaced by
  \begin{eqnarray*}
	-  (t-t_0)\partial_{v_i}^2-(t-t_0)^2 \partial_{x_i}  \partial_{v_i}-\frac{(t-t_0)^3}{3}\partial_{x_i}^2,\quad i=2 \textrm { or } 3,
\end{eqnarray*}
which implies  the validity of \eqref{rev1} with $\partial_{x_1} $  replaced by $\partial_{x_2} $ or $\partial_{x_3} $,  and $\partial_{v_1} $ by $\partial_{v_2} $ or $\partial_{v_3}  $.  As a result, we combine \eqref{rev1} with  the fact
\begin{eqnarray*}
\forall\ \alpha\in\mathbb Z_+^3,\quad 	\norm{\partial_x^\alpha f}_{L^2}\leq \norm{\partial_{x_1}^{\abs \alpha} f}_{L^2}+\norm{\partial_{x_2}^{\abs \alpha} f}_{L^2}+\norm{\partial_{x_3}^{\abs \alpha} f}_{L^2}
\end{eqnarray*}
and similarly for $\norm{\partial_v^\alpha f}_{L^2}$,  to conclude
\begin{equation}\label{muest}
\begin{aligned}
		\forall\ \alpha\in\mathbb Z_+^3,  \quad  \sup_{0<t\leq 1}   \Big(t  ^{\abs\alpha+3}\norm{\partial_{v}^{2\alpha}f(t)}_{(2,0)}+ t ^{3\abs\alpha+3}\norm{\partial_{x}^{2\alpha}f(t)}_{(2,0)}\Big)  \leq  \big( 4 C_0\big)^3 \eps_0  (3C_*/c_0)^{2\abs\alpha}  (2\abs\alpha)!.
\end{aligned}
\end{equation}
Then,   for any $0<t\leq 1$ and   any $\alpha,\beta\in\mathbb Z_+^3$  with $\abs\alpha\geq2$, we can write $\alpha=\tilde \alpha+(\alpha-\tilde\alpha)$ with $|\alpha-\tilde\alpha|=2$ and thus
\begin{align*}
&	t^{\frac{3}{2}\abs\alpha+ \frac{\abs\beta}{2}}\norm{\partial_x^\alpha \partial_{v}^\beta f(t)}_{L^2} \leq t^{\frac{3}{2}\abs{\tilde\alpha}+ \frac{\abs\beta}{2}+3}\norm{\partial_x^{\tilde\alpha} \partial_{v}^\beta f(t)}_{(2,0)} \leq  t^{\frac{3}{2}\abs\alpha+ \frac{\abs\beta}{2}+3} \norm{\partial_x^{2\tilde \alpha }  f(t)}_{(2,0)}^{1/2}\norm{\partial_{v}^{2\beta} f(t)}_{(2,0)}^{1/2}\\
	&\leq  \Big( t^{3\abs\alpha+  3} \norm{\partial_x^{2\tilde \alpha }  f(t)}_{(2,0)}\Big)^{1/2} \Big(t^{ \abs\beta+3} \norm{\partial_{v}^{2\beta} f(t)}_{(2,0)}\Big)^{1/2}\\
	&\leq   \big( 4C_0\big)^3 \eps_0   \Big( (3C_*/c_0)^{2\abs\alpha+2\abs\beta}  (2\abs\alpha)!(2\abs\beta)!\Big)^{1/2}\leq  \big( 4C_0\big)^3 \eps_0    (6C_*/c_0)^{ \abs\alpha+ \abs\beta}  (\abs\alpha+\abs\beta)!,
\end{align*}
the last line using \eqref{muest} and the fact that $p!q!\leq (p+q)!\leq 2^{p+q}p!q!$ for any $p,q\in\mathbb Z_+.$  Meanwhile, for any   $0\leq t\leq 1$ and  any $\alpha,\beta\in\mathbb Z_+^3$  with $\abs\alpha\leq 1,$  we use the fact that $2\abs\alpha\leq 2$ to compute
\begin{align*}
	t^{\frac{3}{2}+  \frac{\abs\beta}{2}}\norm{\partial_x^\alpha \partial_{v}^\beta f(t)}_{L^2}&\leq t^{ \frac{3}{2}+ \frac{\abs\beta}{2}} \norm{ \partial_{v}^{2\beta} f(t)}_{L^2}^{1/2}\norm{ \partial_x^{2  \alpha } f(t)}_{L^2}^{1/2} \leq  \Big( t^{\abs \beta+  3} \norm{\partial_v^{2\beta }  f(t)}_{(2,0)} \Big)^{1/2} \Big( \norm{ f(t)}_{(2,0)}\Big)^{1/2}\\
	&\leq   \big( 4C_0\big)^3 \eps_0 \Big( (3C_*/c_0)^{ 2\abs\beta}   (2\abs\beta)!\Big)^{1/2}\leq   \big( 4C_0\big)^3 \eps_0  (6C_*/c_0)^{   \abs\beta}  ( \abs\beta)!.
\end{align*}
We have proven the assertions \eqref{decayana} and \eqref{decayrate} in Theorem \ref{thm:main} for $0< t\leq 1$ by choosing $C=\max\big\{(4C_0)^3, 6C_*/c_0\big\}$.

 Once the analyticity regularization effect  is achieved for $0\leq t\leq 1$,   it is essentially the propagation of  analyticity  from  $t=1$ to $t>1$ when deriving the analyticity for $t>1$. This will follow  by performing     standard energy estimates for $\norm{\partial_x^\alpha\partial_v^\beta f(t)}_{(2,0)}$ at $t\in ]1,+\infty[$ and there is  no additional difficulty. So we omit it for brevity.  The proof of Theorem \ref{thm:main} is thus completed.
 \end{proof}

\bigskip
\noindent
{\bf Acknowledgments.}
H. M. Cao was supported by the NSFC (No. 12001269) and the Fundamental Research Funds
for the Central Universities of China.
W.-X. Li was supported by the NSFC (Nos. 11961160716, 11871054, 12131017) and the Natural Science Foundation of Hubei Province (No. 2019CFA007).
C.-J. Xu was supported by the NSFC (No.12031006) and the Fundamental Research Funds
for the Central Universities of China.


  \end{document}